\documentclass[]{amsart}
\usepackage{amssymb, mathdots}
\usepackage{mathrsfs}
\usepackage[utf8x]{inputenc}
\usepackage{amsmath,  graphicx}
\usepackage{diagrams}
\usepackage{wrapfig}
\usepackage{enumerate}
\usepackage{url}
\DeclareSymbolFont{bbold}{U}{bbold}{m}{n}
\DeclareSymbolFontAlphabet{\mathbbold}{bbold}
\def\qmod#1#2{{\hbox{}^{\displaystyle{#1}}}\!\big/\!\hbox{}_{
\displaystyle{#2}}}

\def\resto#1#2{{
#1\hskip 0.4ex\vline_{\hskip 0.5ex\raisebox{-0,2ex}
{{${\scriptstyle #2}$}}}}}

\def\at#1#2{
#1\hskip 0.25ex\vline_{\hskip 0.25ex\raisebox{-1.5ex}
{{$\scriptstyle#2$}}}}

\def\A{{\mathbb A}}

\def\C{{\mathbb C}}

\def\L{{\mathbb L}}

\def\P{{\mathbb P}}
\def\Q{{\mathbb Q}}
\def\R{{\mathbb R}}
\def\Z{{\mathbb Z}}

\def\qed {\hfill\vrule height6pt width6pt depth0pt \smallskip}

\def\textmap#1{\mathop{\vbox{\ialign{
                                  ##\crcr
      ${\scriptstyle\hfil\;\;#1\;\;\hfil}$\crcr
      \noalign{\kern 1pt\nointerlineskip}
      \rightarrowfill\crcr}}\;}}
\def\bigtextmap#1{\mathop{\vbox{\ialign{
                                  ##\crcr
      ${\hfil\;\;#1\;\;\hfil}$\crcr
      \noalign{\kern 1pt\nointerlineskip}
      \rightarrowfill\crcr}}\;}}
      
\newcommand{\cal}{\mathcal}
\def\textlmap#1{\mathop{\vbox{\ialign{
                                  ##\crcr
      ${\scriptstyle\hfil\;\;#1\;\;\hfil}$\crcr
      \noalign{\kern-1pt\nointerlineskip}
      \leftarrowfill\crcr}}\;}}
 
\def\ag{{\mathfrak a}}
\def\bg{{\mathfrak b}}

\def\hg{{\mathfrak h}}

\def\mg{{\mathfrak m}}

\def\rg{{\mathfrak r}}

\def\ug{{\mathfrak u}}

\def\Ag{{\mathfrak A}}
\def\Bg{{\mathfrak B}}

\def\Hg{{\mathfrak H}}
\def\Ig{{\mathfrak I}}
\def\Jg{{\mathfrak J}}

\def\Mg{{\mathfrak M}}
\def\Ng{{\mathfrak N}}

\def\Rg{{\mathfrak R}}

\def\Sg{{\mathfrak S}}
\def\Ug{{\mathfrak U}}

\newtheorem{sz}{Satz}[section]
\newtheorem{thry}[sz]{Theorem}
\newtheorem{pr}[sz]{Proposition}
\newtheorem{re}[sz]{Remark}
\newtheorem{co}[sz]{Corollary}
\newtheorem{dt}[sz]{Definition}
\newtheorem{lm}[sz]{Lemma}

\def\End{\mathrm {End}}

\def\SU{\mathrm {SU}}

\def\GL{\mathrm {GL}}
\def\SL{\mathrm {SL}}

\def\Pic{\mathrm {Pic}}

\def\NS{\mathrm{NS}}
\def\deg{\mathrm {deg}}

\def\vol{\mathrm{vol}}

\def\id{ \mathrm{id}}
\def\im{\mathrm{im}}

\def\si{\mathrm {si}}
\def\st{\mathrm {st}}
\def\pst{\mathrm{pst}}
\def\ev{\mathrm{ev}}

\def\p{\mathrm{p}}

\def\ASD{\mathrm{ASD}}
\def\HE{\mathrm{HE}}

\newcommand\smvee{{\hskip -0.3ex \raise 0.2ex\hbox{$\scriptscriptstyle\vee$}}}
\newcommand{\longhookrightarrow}{}
\DeclareRobustCommand{\longhookrightarrow}{\lhook\joinrel\relbar\joinrel\rightarrow}
    
\begin{document}
 
 \title[Holomorphic models around the reduction loci]{Instanton moduli spaces on non-Kählerian surfaces. Holomorphic models around the reduction loci}
 \author{Andrei Teleman} 
 \thanks{The author has been partially supported by the ANR project MNGNK, decision 
Nr.  ANR-10-BLAN-0118}
\address{Aix Marseille Université, CNRS, Centrale Marseille, I2M, UMR 7373, 13453 Marseille, France }

\begin{abstract}
Let $\mathcal{M}$ be a moduli space of  polystable rank 2-bundles   bundles with fixed determinant (a moduli space of $\mathrm{PU}(2)$-instantons) on a Gauduchon surface with $p_g=0$ and $b_1=1$. We study the holomorphic structure of  $\mathcal{M}$ around a circle $\mathcal{T}$ of regular reductions.  Our model space is a "blowup flip passage", which is a manifold with boundary whose boundary  is a projective fibration, and whose interior comes with a natural complex structure.    

We prove that a neighborhood of the boundary of the blowup $\hat{\mathcal{M}}_{\mathcal{T}}$ of $\mathcal{M}$  at $\mathcal{T}$ can be smoothly identified with a neighborhood of the boundary of a "flip passage" $\hat Q$, the identification being holomorphic on $\mathrm{int}(\hat Q)$. 
\end{abstract}

\maketitle 
 \section{Introduction} \label{intro}
 
 \subsection{Tori of reductions in instanton moduli spaces} \label{intro1} 

Let $(M,g)$ be a closed, connected, oriented Riemannian 4-manifold, and $(E,h)$ be a Hermitian bundle on $M$. Let $a$ be a Hermitian connection on the Hermitian line bundle $D:=\det(E)$,  denote by ${\cal A}(E)$ be the space of Hermitian connections on $E$, and put
$${\cal A}_a(E):=\{A\in {\cal A}(E)|\ \det(A)=a\}\ ,\ {\cal G}_E:=\Gamma(X,\SU(E))\ ,\ {\cal B}_a(E):=\qmod{{\cal A}_a(E)}{{\cal G}_E}$$
$$
{\cal A}^\ASD_a(E):=\{A\in {\cal A}_a(E)|\ (F_A^0)^+=0\}\ ,\ {\cal M}_a^{\ASD}(E):=\qmod{{\cal A}_a^\ASD(E)}{{\cal G}_E}.$$
Using a well known slice theorem (see for instance \cite{DK}), one can prove that the infinite dimensional quotient ${\cal B}_a(E)$, endowed with the quotient topology, is Hausdorff. Its subspace ${\cal M}_a^{\ASD}(E)$  is finite dimensional and will be called the moduli space  of $a$-oriented  projectively ASD connections on $(E,h)$.  The open subspace ${\cal B}^*_a(E)$ of ${\cal B}_a(E)$ defined by the condition ``$A$ is irreducible" (or equivalently ``the  stabilizer of $A$  with respect to the ${\cal G}_E$-action is $\{\pm\id_E\}$) becomes a real analytic Banach manifold after suitable Sobolev completions, and the corresponding subspace
 ${\cal M}_a^{\ASD}(E)^*\subset {\cal M}_a^{\ASD}(E)$  has the structure of a finite dimensional real analytic space.  The reduction locus ${\cal M}_a^{\ASD}(E)\setminus {\cal M}_a^{\ASD}(E)^*$ can be described as follows.\\
 
 The set of equivalence classes of  decompositions of $E$ (as orthogonal direct sum of line subbundles) can be identified with the quotient
$${\cal D}ec(E):=\qmod{\{c\in H^2(X,\Z)|\ c(c_1(E)-c)=c_2(E)\} }{\iota}
$$ 
where $\iota$ is the involution $c\mapsto c':=c_1(E)-c$. Suppose $b_+(M)=0$ and let $\lambda=\{c,c'\}\in {\cal D}ec(E)$ with $c\ne c'$. The moduli subspace 
$${\cal T}_\lambda:=\{[A]\in {\cal M}_a^\ASD(E)|\ E \hbox{ has an $A$-parallel line bundle $L$ with }  c_1(L)\in \lambda \}
$$ 
 of ${\cal M}_a^\ASD(E)$  is a torus of dimension $b_1(M)$, which will be called the torus of $\lambda$-reductions.  If   $c_1(E)\not\in 2H^2(X,\Z)$, then for every  $\lambda=\{c,c'\}\in {\cal D}ec(E)$ one has $c\ne c'$ and any reduction $A\in{\cal A}_a^{\ASD}(E)\setminus {\cal A}_a^{\ASD}(E)^*$ will be abelian (has $S^1$ as stabilizer). In this case one has 
 \begin{equation}\label{DecRed}
{\cal M}_a^{\ASD}(E)\setminus {\cal M}_a^{\ASD}(E)^*=\coprod_{\lambda\in  {\cal D}ec(E)} {\cal T}_\lambda\ ,
 \end{equation}
so  the reduction locus of ${\cal M}_a^{\ASD}(E)$ is a disjoint union  of tori of dimension $b_1(M)$.

Suppose that $b_+(M)=0$ and  ${\cal T}_\lambda$ is a torus of {\it regular reductions}, i.e. the second cohomology space of the deformation elliptic complex of $A$ vanishes for every $[A]\in {\cal T}_\lambda$.  Then one can define the blowup $\hat {\cal M}_a^{\ASD}(E)_\lambda$ of ${\cal M}_a^{\ASD}(E)$ at the torus ${\cal T}_\lambda$ (see \cite{Te3} section 1.4.2).  By  construction $\hat {\cal M}_a^{\ASD}(E)_\lambda$ has the following properties:
\begin{itemize}
\item comes with a proper surjective map 
$$p_\lambda:\hat {\cal M}_a^{\ASD}(E)_\lambda\to {\cal M}_a^{\ASD}(E)$$
which induces an isomorphism  $\hat {\cal M}_a^{\ASD}(E)_\lambda\setminus p_\lambda^{-1}({\cal T}_\lambda)\to {\cal M}_a^{\ASD}(E)\setminus {\cal T}_\lambda$,
\item has the structure of a smooth manifold with boundary ${\cal P}_\lambda:=p_\lambda^{-1}({\cal T}_\lambda)$ around ${\cal P}_\lambda$,
\item the induced map $\pi_\lambda:{\cal P}_\lambda\to {\cal T}_\lambda$  is the projectivization of a complex vector bundle over ${\cal T}_\lambda$ .
\end{itemize}

The existence of the blowup moduli space $\hat {\cal M}_a^{\ASD}(E)_\lambda$ with the above properties has important consequences:
\begin{enumerate}
\item Using this result, one can prove easily that a torus of regular reductions has a neighborhood $N_\lambda$ which is a locally trivial fiber bundle over ${\cal T}_\lambda$  whose fiber is a cone over a complex projective space, and such that ${\cal T}_\lambda$ corresponds to the vertex section of this cone bundle. In other words  we have a simple {\it topological model} of ${\cal M}_a^\ASD(E)$ around a torus ${\cal T}_\lambda$ of regular reductions: a cone bundle over ${\cal T}_\lambda$ whose fiber is a cone over a complex projective space.
\item Blowing up all tori of reductions ${\cal T}_\lambda$ (which we assume to be regular) in ${\cal M}_a^{\ASD}(E)$ one obtains a blowup moduli space  $\hat {\cal M}_a^{\ASD}(E)$ on which all Donaldson classes $c\in H^*({\cal B}^*_a(E),\Q)$ extend (see \cite{Te5}). In particular, for any homology class $h\in H_*(\hat {\cal M}_a^{\ASD}(E),\Q)$ the evaluation $\langle c,h\rangle$ has sense.

\end{enumerate}

\subsection{Gauduchon stability and the Kobayashi-Hitchin correspondence on non-Kählerian surfaces. The problem} \label{intro2}

 We refer to \cite{LT1} for the general stability theory   on arbitrary compact Gauduchon manifolds. In this article we focus on the situation which is relevant from the point of view of Donaldson theory: rank two bundles on complex surfaces.
 
 Let $X$ be a complex surface endowed with a {\it Gauduchon metric} $g$ \cite{Gau}. A holomorphic rank 2-bundle ${\cal E}$ on $X$ is called
\begin{itemize}
\item  $g$-{\it stable},  if for every line bundle ${\cal L}$ and   sheaf monomorphism ${\cal L}\to {\cal E}$ one has
$\deg({\cal L})<\frac{1}{2}\deg_g(\det({\cal E}))$.
\item  $g$-{\it polystable}, if is either stable, or isomorphic to a direct sum ${\cal L}\oplus {\cal M}$  of line bundles with $\deg_g({\cal L})=\deg_g({\cal M})$.
\end{itemize}

 Let  $E$ be a differentiable   rank 2-bundle on $X$, ${\cal D}$ a fixed holomorphic structure on  $D:=\det(E)$. We denote by ${\cal M}^\st_{\cal D}(E)$, ${\cal M}^\pst_{\cal D}(E)$ the moduli space  of stable (respectively polystable) holomorphic structures on $E$ inducing   ${\cal D}$ on $\det(E)$, modulo the complex gauge group ${\cal G}^\C_E:=\Gamma(X,\SL(E))$ (see the section \ref{semi} of the Appendix).
 ${\cal M}^\st_{\cal D}(E)$ has a natural complex space structure obtained using its open embedding in the corresponding moduli space ${\cal M}^\si_{\cal D}(E)$ of simple holomorphic structures. ${\cal M}^\si_{\cal D}(E)$ is a finite dimensional, but in general non-Hausdorff, complex space \cite{LO}. On the other hand   ${\cal M}^\pst_{\cal D}(E)$ is Hausdorff.   The Hausdorff property of  ${\cal M}^\pst_{\cal D}(E)$   is   a consequence of the Kobayashi Hitchin correspondence, which we recall briefly in our framework.\\
 
Let $h$ be a Hermitian metric on $E$ and let $a$ be the Chern connection of the pair $({\cal D},\det(h))$. The {  Kobayashi-Hitchin correspondence} states that the map 
$$A {\mapsto}\hbox{ the holomorphic structure defined by }  \bar\partial_A$$
induces a   homeomorphism $KH:{\cal M}_a^{\ASD}(E)\to {\cal M}^\pst_{\cal D}(E)$ which restricts to a real analytic isomorphism ${\cal M}_a^{\ASD}(E)^*\to  {\cal M}^\st_{\cal D}(E)$. More precisely we have a commutative diagram 
$$
\begin{array}{c}
\unitlength=1mm
\begin{picture}(50,50)(-42,-38)
\put(-46,8){${\cal B}_a^*(E)$}
\put(-34,8){$\supset$}
\put(20,8){${\cal B}_a(E)$}
\put(15,8){$\subset$}

\put(-30,8){${\cal M}_a^{\ASD}(E)^*$}
\put(-11,8){$\longhookrightarrow$}
\put(-2,8){${\cal M}_a^{\ASD}(E)$}
\put(-24,4){\vector(0, -3){5}}
\put(-31,1){$\scriptstyle KH^*$} \put(-22,1){$\simeq$}
\put(5,4){\vector(0, -3){5}}
\put(6,1){$\scriptstyle KH$} \put(0,1){$\simeq$}
\put(-34,-7){$\supset$}
\put(-48,-7){${\cal M}_{\cal D}^{\rm si}(E)$}
\put(-29,-7){${\cal M}_{\cal D}^\st(E)$}
\put(-12,-7){$\longhookrightarrow$}
\put(-2,-7){${\cal M}_{\cal D}^\pst(E)$ ,}

\end{picture} 
\end{array} 
$$
\vspace{-32mm}\\
where $KH$ is a homeomorphism and $KH^*$ a real analytic isomorphism.

In general ${\cal M}_{\cal D}^\pst(E)$ is not a complex space around {\it the reduction locus} ${\cal R}:={\cal M}_{\cal D}^\pst(E)\setminus {\cal M}_{\cal D}^\st(E)$ (\cite{Te2}, \cite{Te3}), which  can be identified via $KH$ with the subspace of reducible instantons in ${\cal M}_a^{\ASD}(E)$. Under the assumption $c_1(E)\not\in 2 H^2(M,\Z)$) this subspace has been described above in the general gauge theoretical framework.\\

Suppose now that $p_g(X)=0$ and $b_1(X)=1$.  Such a surface is non-Kählerian and has  $b_+(X)=0$.
In previous articles (\cite{Te2}-\cite{Te4}) we have shown that studying  certain  moduli spaces of the form ${\cal M}_{\cal D}^\pst(E)={\cal M}^\ASD_a(E)$ on such surfaces is interesting, important and difficult.  For instance, using a combination of complex geometric and gauge theoretical techniques, we proved geometric properties of  certain moduli spaces of this type and we used them to prove existence of curves on class VII surfaces with small $b_2$. 

An important role in our arguments was played by {\it the circles of reductions} ${\cal T}_\lambda$. Whereas the blowup construction explained above gives a precise    topological description  of  ${\cal M}_a^\ASD(E)$ around a circle ${\cal T}_\lambda$ of regular reductions, understanding the complex structure of ${\cal M}_a^\ASD(E)^*$ around such a circle  is a much more difficult and interesting problem.

The goal of this article is to address this problem, hence to describe explicitly by means of {\it holomorphic models} the complex structure of the end of the moduli space ${\cal M}_a^\ASD(E)^*$ (equivalently of ${\cal M}^\st_{\cal D}(E)$)  towards   a circle ${\cal T}_\lambda$ of regular reductions. More precisely we will  construct 
\begin{itemize}
\item an explicit manifold with boundary ${\cal Q}_\lambda$ with a complex  structure on its interior and whose boundary is a projective bundle over ${\cal T}_\lambda$,
\item a diffeomorphism from ${\cal Q}_\lambda$ onto an open neighborhood $O_\lambda$ of the boundary ${\cal P}_\lambda$ of   $\hat{\cal M}_a^\ASD(E)_\lambda$, which induces  a biholomorphism  $\mathrm{int}({\cal Q}_\lambda)\to \mathrm{int}(O_\lambda)$ and   a bundle isomorphism $\partial{\cal Q}_\lambda\to {\cal P}_\lambda$ over ${\cal T}_\lambda$.
\end{itemize}

Why are we interested in  holomorphic models for the  ends of ${\cal M}_a^\ASD(E)^*={\cal M}^\st_{\cal D}(E)$? The reason is the following problem, which plays an important role in our program for proving existence of curves on class VII surfaces:
\vspace{2mm}\\
{\bf Problem:}   {\it  Suppose $c_1(E)\not\in 2 H^2(X,\Z)$, let $d=4c_2(E)-c_1(E)^2$ be the expected complex dimension of ${\cal M}^\st_{\cal D}(E)$,   and let $Z\subset {\cal M}^\st_{\cal D}(E)$ be a pure $k$ dimensional analytic set with $1\leq k\leq d$.  Determine explicitly the boundary $\delta_\lambda([Z]^{BM})\in H_{2k-1}({\cal P}_\lambda,\Z)\simeq\Z$ of the  Borel-Moore fundamental class of $Z$. } 
\\

Intuitively $\delta_\lambda([Z]^{BM})$ is obtained by intersecting the closure of $Z$ in $\hat {\cal M}^\ASD_a(E)$ with the boundary component ${\cal P}_\lambda$.
Suppose that   ${\cal M}_a^\ASD(E)$ is compact  and all reductions in this moduli space are  regular. These conditions are  satisfied for the moduli spaces studied in \cite{Te2}-\cite{Te3}. Then, for any Donaldson class $c\in H^{2k-1}({\cal B}^*_a,\Q)$ we  will have
\begin{equation}\label{sum}\sum_{\lambda\in{\cal D}ec(E)} \langle \resto{c}{{\cal P}_\lambda}, \delta_\lambda([Z]^{BM})\rangle =0\ .
\end{equation}
The restrictions $\resto{c}{{\cal P}_\lambda}$ have been computed explicitly \cite{Te5}. Therefore, supposing that  the problem above has been solved,   (\ref{sum}) can be interpreted as  an obstruction  to the existence of analytic sets $Z\subset {\cal M}^\st_{\cal D}(E)$  with prescribed topological behavior around the circles of reductions.  We will make use of these ideas in a future article. 
\subsection{Flip passages and the holomorphic model theorem}\label{flips}

 Let $V'$, $V''$ be Hermitian spaces of dimensions $r'$, $r''$.  We let $\C^*$ act on $V'\times V''$ by 
$$\zeta\cdot(y',y''):=(\zeta y',\zeta^{-1} y'')\ .$$ 
The  induced  $S^1$-action has a 1-parameter family $(\mu_t)_{t\in\R}$ of moment maps   given by $\mu_t=-im_t$, where $m_t(y',y'')=\frac{1}{2}(\|y'\|^2-\|y''\|^2)+t.$ The corresponding family of Kähler quotients is:
$$Q_t:=\qmod{m_t^{-1}(0)}{S^1}=\left\{
\begin{array}{ccc}
Q':= {V'_*\times V''}/{\C^*} &\rm for & t<0\\  
Q_0:= \left\{({V'_*\times V''_*})\cup\{0\}\right\}/{\C^*}  &\rm for & t=0\\ 
Q'':= {V'\times  V''_*}/{\C^*} &\rm  for & t>0
\end{array}
\right. \ .
$$
Denoting by  $\Theta'$, $\Theta''$ be the tautological line bundles over $\P(E')$, $\P(E'')$ we get obvious biholomorphisms:
$$Q'\stackrel{\rm bihol}{\simeq}\{\Theta'\otimes V''\}\stackrel{\rm bihol}{\simeq}\{\Theta'\}^{\oplus r''} \ ,\ Q''\stackrel{\rm bihol}{\simeq}\{\Theta''\otimes V'\}\stackrel{\rm bihol}{\simeq}\{\Theta''\}^{\oplus r'}$$
 $Q_0$ can be identified with the image of the natural map $V'\times V''\to V'\otimes V''$, and (if $r'>0$, $r''>0$) it has an isolated singularity. It will be called {\it the singular quotient} and can be identified with the cone (in the algebraic geometric sense) over the image of the Segre embedding  $\P(V')\times\P(V'')\to \P(V'\otimes V'')$.
We also define {\it the   blowup quotient} $\tilde Q$ by  
$$\tilde Q:={p'}^*(\Theta')\otimes {p''}^*(\Theta'')\ ,$$
which is a holomorphic line bundle over $\P(V')\times\P(V'')$. These spaces fit in the commutative diagram
\begin{diagram}[h=5mm]
&&&{ \P(V')\times\P(V'')} &&&\\
&&&\dInto &&&\\
&&&\tilde Q&&&\\
&&\ldTo &&\rdTo &&\\
{ \P(V')}\subset &Q'&\ \ \ \  \leftarrow &  \hbox{\Large a flip}  &\rightarrow \ \ \ \ &Q''&\supset { \P(V'')}\\
&&\rdTo &&\ldTo &&\\
&&&Q_0&&&\\
&&&\uInto &&&\\
&&&\{{ *}\} &&& 
\end{diagram}
in which the four maps define biholomorphisms between the complements of the corresponding closed subspaces appearing in the diagram. The birational map $Q'\dashrightarrow Q''$ is a standard example of a flip (see \cite{R}).\\

Let $p':E'\to B$, $p'':E''\to B$ be holomorphic bundles of ranks $r'$, $r''$ endowed with Hermitian metrics $h'$, $h''$ on a complex manifold $B$. Endow $E:=E'\oplus E''$ with the $\C^*$-action 
$$\zeta\cdot(y',y'')= (\zeta y',\zeta^{-1}y'')\ .$$
Let $f:B\to\R$ a smooth map which is a submersion at any vanishing point. Therefore the zero set $T:=Z(f)$ is an oriented real hypersurface of $B$. We assume that $T$ is compact. We define the smooth family of fiberwise moment maps  associated with $f$ by $\mu^f=-im^f:E\to i\R$, where  $m^f$ is given fiberwise by
$$m_b^f(y',y'')=\frac{1}{2}(\| y'\|^2-\|y''\|^2)+f(b)\ , \ \forall (y',y'')\in E_b \ .
$$
 Put 
 $$Q_f:=\qmod{Z(f)}{S^1}\ ,\ Q^*_f:=\qmod{Z(m)\setminus Z(m)^{S^1}}{S^1}\ .$$
One has an obvious identification (induced by $p$)  
$\{Z(m)\}^{S^1}=T$ . $Q^*_f$ comes with an obvious  open embedding  $Q^*_f\hookrightarrow \{E\setminus B\}/\C^*$, hence is naturally a complex manifold of dimension  $r'+r''+\dim(B)-1$. Therefore $Q_f$ is obtained by adding the real hypersurface $T$   to the  complex manifold $Q^*_f$. Note that $Q_f$ comes with a map $Q_f\to B$ whose fibers are 
$$\{Q_f\}_{b}\simeq \left\{
\begin{array}{ccc} Q' &\rm for & f(b)<0\\
Q_0 &\rm for &f(b)=0\\
 Q'' &\rm for & f(b)>0
\end{array}\right. \ .
$$
Therefore $Q_f$ should be called the {\it flip passage} from $Q'$ to $Q''$ (passing through $Q_0$) associated with the system $(B,E',E'',h',h'',f)$. The complement $T=Q_f\setminus Q^*_f$ is formed by the singularities of the fibers of type $Q_0$.  The {\it blowup flip passage} $\hat Q_f$ is defined by $\hat Q_f:=\widehat{m_f^{-1}(0)}/S^1$, where $\widehat{m_f^{-1}(0)}$ is the spherical blowup of the smooth hypersurface $m_f^{-1}(0)\subset E$ at the $S^1$-fixed point locus.  $\hat Q_f$ is a manifold with boundary, whose interior coincides with the complex manifold ${\cal Q}^*_f$ and whose boundary $\partial \hat Q_f$ can be identified with the projective bundle $\P(\resto{E'}{T}\oplus \resto{\bar E''}{T})$ (see section \ref{BlUpS1Q} for details).\\

Let now $(X,g)$ be a Gauduchon surface with $p_g=0$, $b_1(X)=1$. For such a surface one has $\Pic^0(X)\simeq\C^*$ and, choosing an isomorphism $\C^*\ni\zeta\mapsto {\cal L}_\zeta\in \Pic^0(X)$  in a convenient way, one has
$$\deg_g({\cal L}_\zeta)=C_g\ln|\zeta|\ ,
$$
for a positive constant $C_g$ depending smoothly on $g$. This shows that the level sets of the restriction of $\deg_g$ on any component of $\Pic(X)$ are circles.

Let $(E,h)$  be a Hermitian rank 2-bundle on $X$, ${\cal D}$ a fixed holomorphic structure on $D:=\det(E)$ and $a$ the Chern connection of the pair $({\cal D},\det(h))$. Let $L$ be a line subbundle of $E$ with $2c_1(L)\ne c_1(E)$ and  $\lambda=\{c_1(L),c_1(E)-c_1(L)\}$ the corresponding element of ${\cal D}ec(E)$. Suppose that   ${\cal T}_\lambda\subset {\cal M}_a^\ASD(E)$ is a circle of {\it regular} reductions. 

Fix $x_0\in X$ and let $\mathscr{L}=\mathscr{L}_{x_0}$ be the Poincaré line bundle  associated with the base point $x_0$ (see Definition \ref{Poincare}, section \ref{GaugePic}) on $\Pic^c(X)\times X$ endowed with its canonical Hermitian metric (see Remark \ref{metric}).  Denoting by $p_1:\Pic^c(X)\times X\to \Pic^c(X)$, $p_2:\Pic^c(X)\times X\to X$ the two projections, put
$${\cal H}':=R^1(p_1)_*\big(\mathscr{L}^{\otimes 2}\otimes p_2^*({\cal D}^\vee)\big)\ ,\ {\cal H}'':=R^1(p_1)_*\big(\mathscr{L}^{-\otimes 2}\otimes p_2^*({\cal D})\big)\ ,
$$
and let $f_{\cal D}:\Pic^c(X)\to\R$ be the harmonic map defined by
$$f_{\cal D}([{\cal L}]):=\pi\big(\deg_g({\cal L})-\frac{1}{2}\deg_g({\cal D})\big)\ ,
$$

The vanishing circle $T:=Z(f_{\cal D})$ can be identified  with the reduction torus ${\cal T}_\lambda\subset {\cal M}^\pst_{\cal D}(E)$ using the map $k:T\to {\cal T}_\lambda$ given by
$$k_\lambda({\cal L}):=[{\cal L}\oplus ({\cal D}\otimes {\cal L}^\vee)]\ .
$$ 
 For $\varepsilon>0$  consider the annulus $\Pic^c_\varepsilon(X):=f_{\cal D}^{-1}(-\varepsilon,\varepsilon)$.
Under our assumptions  it follows that, for any sufficiently small $\varepsilon>0$, the restrictions 
$$\resto{{\cal H}'}{\Pic^c_\varepsilon(X)}\ ,\ \resto{{\cal H}''}{\Pic^c_\varepsilon(X)}
$$
are locally free of ranks
$$r'=-\frac{1}{2} (2c -c_1(E))(2c -c_1(E)+c_1(X))\ ,\ r''=-\frac{1}{2} (-2c +c_1(E))(-2c +c_1(E)+c_1(X))
$$
respectively, and for any $l\in \Pic^c_\varepsilon(X)$ one has canonical identifications 
$${\cal H}'(l)=H^1\big(\mathscr{L}^{\otimes 2}_l\otimes p_2^*({\cal D}^\vee)\big)\ ,\ {\cal H}''(l)=H^1\big(\mathscr{L}^{\otimes -2}_l\otimes p_2^*({\cal D})\big)\ .$$
We will see that, for $l\in T$, these spaces come with natural Hermitian products obtained by identifying them with suitable harmonic spaces (see Proposition \ref{H1}), which will be endowed with   standard $L^2$-Hermitian products. In this way we get Hermitian metrics $\hg'$, $\hg''$ on the bundles $\resto{{\cal H}'}{T}$, $ \resto{{\cal H}''}{T}$, which can be extended to get Hermitian metrics $h'$, $h''$ on
$\resto{{\cal H}'}{\Pic^c_\varepsilon(X)}$, $ \resto{{\cal H}''}{\Pic^c_\varepsilon(X)}$.
Let $Q$ ($\hat Q$) be  the (blowup)  flip passage associated with the system
$$\big (\Pic^c_\varepsilon(X), \resto{{\cal H}'}{\Pic^c_\varepsilon(X)}, \resto{{\cal H}''}{\Pic^c_\varepsilon(X)}, h', h'', \resto{f_{\cal D}}{\Pic^c_\varepsilon(X)}\big)\ .$$
 Our holomorphic model theorem states:
 \begin{thry} \label{hol-model}
  
Under the assumptions and with the notations above there exists an open neighborhood $O$ of $\partial\hat Q$ in $\hat Q$ and a diffeomorphism $\chi:O\to {\cal O}_\lambda$ onto a smooth open neighborhood ${\cal O}_\lambda$ of   ${\cal P}_\lambda$ in the blow up moduli space $\hat{\cal M}^\ASD_a(E)_\lambda$ such that
\begin{enumerate}[1.] 
\item  $\chi$ induces a smooth bundle isomorphism 
\begin{equation}\label{bundle-iso}\begin{diagram}[h=7mm]
\P\big(\resto{{\cal H}'}{T}\oplus \resto{\bar {\cal H}''}{T}\big)= \partial\hat Q  &\rTo^{\partial \chi} & {\cal P}_\lambda\\
\dTo & &\dTo^{\pi_\lambda}\\
T&\rTo^{k_\lambda} & {\cal T}_\lambda \ ,
\end{diagram}
\end{equation}
 
\item $\chi$  induces a biholomorphism $O\setminus \partial\hat Q\to {\cal O}_\lambda\setminus {\cal P}_\lambda$.
\end{enumerate}

 \end{thry}

Therefore, around the boundary ${\cal P}_\lambda$, the blowup moduli space   $\hat{\cal M}^\ASD_a(E)_\lambda$ can be identified with a neighborhood of the boundary of the  blowup flip passage associated with the system $\big (\Pic^c_\varepsilon(X), \resto{{\cal H}'}{\Pic^c_\varepsilon(X)}, \resto{{\cal H}''}{\Pic^c_\varepsilon(X)}, h', h'', \resto{f_{\cal D}}{\Pic^c_\varepsilon(X)}\big)$, and this identification respects the complex structure on $ {\cal M}^\ASD_a(E)^*$ induced by the Kobayashi-Hitchin correspondence. The difficult part of the result is the holomorphy property of our model. Note that the holomorphic bundles $\resto{{\cal H}'}{\Pic^c_\varepsilon(X)}, \resto{{\cal H}''}{\Pic^c_\varepsilon(X)}$ are in fact trivial. 

Collapsing to points the fibers of the two projective fibrations in (\ref{bundle-iso}), we obtain the following {\it weaker} holomorphic model theorem:

\begin{co} There exists a homeomorphism $U\to U_\lambda$ from an open neighborhood $U$ of $T$ in  $Q_f$ onto an open neighborhood $U_\lambda$ of ${\cal T}_\lambda$ in ${\cal M}^\ASD_a(E)$, which restricts to the standard identification $k_\lambda:T\to {\cal T}_\lambda$ and a biholomorphism $U\setminus T\to U_\lambda\setminus {\cal T}_\lambda$.
\end{co}

\section{$\C^*$-quotients of holomorphic bundles with respect to families of moment maps}

\subsection{Blowup $S^1$-quotients and families of $S^1$-moment maps on Hermitian bundles}\label{BlUpS1Q}

Let $Z$ be a differentiable manifold endowed with an $S^1$-action such that the stabilizer of any point $z\in Z$ is either $S^1$ or trivial. This implies that the fixed point locus $Z^{S^1}$ is a submanifold of $Z$, and that the normal bundle $N_{Z^{S^1}}^Z$ of $Z^{S^1}$ in $Z$ has a complex structure such that the induced action of $S^1$ is the standard one. 
 The induced $S^1$-action one the spherical blowup  \cite{AK} $\hat Z_{Z^{S^1}}$ of $Z$ at  $Z^{S^1}$  is free, and the corresponding quotient has a natural structure of  a manifold with boundary, whose boundary can be identified with $\P(N_{Z^{S^1}}^Z)$. 
 
 \begin{dt} \label{BlUpQ}
 The quotient $\hat Z_{Z^{S^1}}/S^1$, endowed with its natural structure of  a manifold with boundary,  will be called the blowup $S^1$-quotient of $Z$ and will be denoted by $\hat Z/S^1$.
 \end{dt}

Let $p':E'\to B$, $p'':E''\to B$ be the projection maps of two holomorphic vector bundles over a (connected) complex manifold $B$, and $p:E:=E'\times_B E''\to B$  the projection map of their direct sum $E'\oplus E''$. We let $\C^*$ act on $E$ by
$$\zeta(u',u''):=(\zeta u', \zeta^{-1} u'')\ .
$$
Identify  in the obvious way the base $B$ with the image $\{(0'_b,0''_b)|\ b\in B\}\subset E$ of the zero section  of $E$. 
\begin{re}
The induced $\C^*$-action on $E^*:=E\setminus B$ is free. The quotient $E^*/\C^*$ has the structure of a   complex manifold which makes the canonical projection 
$$E^*\to  E^*/\C^*$$
 a holomorphic submersion. This manifold is non-Hausdorff if $r'>0$ and $r''>0$. 
\end{re}
We can obtain a Hausdorff $\C^*$-quotient  using ideas from the theory of Kählerian quotients (see for instance \cite{Te1}). We will not use a moment map for the  induced $S^1$-action on the total space $E$, {\it but a family of fiberwise moment maps parameterized by the base $B$}. Fix Hermitian metrics $h'$, $h''$ on $E'$, $E''$ respectively, and endow every fiber $E_b=E'_b\times E''_b$ with the corresponding product Kähler metric.   
\begin{dt}
Let be $f:B\to \R$   a smooth map. The family of fiberwise moment maps  associated with $f$ is the smooth map $\mu^f=-im^f:E\to i\R$, where  $m^f$ is given fiberwise by
$$m_b^f(y',y'')=\frac{1}{2}(\| y'\|^2-\|y''\|^2)+f(b)\ , \ \forall (y',y'')\in E_b \ .
$$
\end{dt}
The  fiberwise stable locus associated with $\mu^f$ is the open set
$$E^\st_f:=\left\{y=(y',y'')\in E\ \vline\ \left\{\begin{array}{ccc}
y'\ne 0 &\rm if & f(p(y))<0\\
y''\ne 0 & \rm if & f(p(y))>0\\
y'\ne 0 \hbox{ and }y''\ne 0& \rm if & f(p(y))=0
\end{array}
  \right.  \right\}\subset E\ ,
$$
and the quotient ${E^\st_f}/{\C^*}$ is a smooth  complex manifold of dimension $\dim(B)+r'+r''-1$. This quotient is Hausdorff because, as explained in section \ref{flips},  it can be identified with the open subspace $Q^*_f:=\{Z(m^f)\setminus Z(m^f)^{S^1}\}/{S^1}$ of the $S^1$-quotient
$$Q_f:=\qmod{ Z(m^f)}{S^1}
$$
which is Hausdorff, as the quotient of a Hausdorff space by a compact group.  

Recall that in the theory of Kähler quotients one is given a holomorphic action of a complex reductive group $G$ on a Kähler manifold whose Kähler structure is invariant under a maximal compact subgroup $K\subset G$, and a moment map $\mu$ for the induced $K$-action. In this framework the $K$-quotient $Z(\mu)/K$ has the structure of a complex space which can be identified with the {\it good}  $G$-quotient  of the semistable locus (see \cite{Te1}). The analogous property does not hold in our situation, hence the complex structure of $Q^*_f$ might not extend to  $Q_f={ Z(m^f)}/{S^1}$. This happens  because the map $\mu^f=-im^f$ is not a moment map for the $S^1$-action on $E$, but just a smooth family of fiberwise moment maps. The complement   $Q_f\setminus Q^*_f $ can be identified with the fixed point locus $Z(m^f)^{S^1}$, which can be further identified with the zero locus $Z(f)\subset B$. 

  Suppose that $f$ is a submersion at any vanishing point. This implies that the zero locus $T:=Z(f)$ is a smooth, {\it oriented} real hypersurface of $B$. $T$ is oriented in the standard way. With this convention, the oriented manifold $\pm T$ is the boundary of the oriented manifold with boundary $T_\pm:=(\pm f)^{-1}([0,\infty))$.

  Moreover this also implies $m^f$ is a submersion at any   vanishing point, so that    $Z(m^f)$ becomes a smooth smooth, oriented real hypersurface of $E$ which is endowed with an $S^1$-action and intersects the zero section $B$ transversally along the fixed point locus $Z(f)$. This $S^1$-action is free away of this fixed point locus. Therefore, we are precisely in the situation appearing in Definition \ref{BlUpQ}, in which we introduced the concept of a blowup $S^1$-quotient.
  
\begin{dt}  Suppose that $f$ is a submersion at any vanishing point and that the real hypersurface $T=Z(f)$ is compact. The (blowup) flip passage associated with the system $(B,E',E'',h',h'',f)$ is the quotient $Q_f:= { Z(m^f)}/{S^1}$ (respectively the blowup $S^1$-quotient   $\hat Q_f:= {\widehat{Z(m^f)}}/{S^1}$).
\end{dt} 
Therefore, by definition, the blowup flip passage $\hat Q_f$ is a manifold with boundary 
$$\partial  \hat Q_f=\P\big(N_{T}^{Z(m^f)}\big)\ ,
$$
and its interior can be identified with $Q_f^*$, hence it comes with a natural complex structure. Putting
$$F':=\resto{E'}{T}\ ,\ F'':=\resto{E''}{T}\ , 
$$
we have an isomorphism of $S^1$-bundles over $T$
$$\resto{T_{Z(m^f)}}{T}=T_T\oplus (F'\oplus\bar F'')\ ,\  N_{T}^{Z(m^f)}=F'\oplus \bar F''\ ,\ \partial  \hat Q_f=\P(F'\oplus \bar F'')\ .
$$
 Note that the boundary orientation of $\partial  \hat Q_f$ (with respect to the complex orientation of the interior of $\hat Q_f$) does not coincide in general with the natural orientation of the projective fibration $\P(F'\oplus \bar F'')$ over the oriented manifold $T$. The two orientations can be compared easily.

\subsection{Perturbations of $m^f$}\label{pert}

Let  $(B,p':E'\to B, p'':E''\to B, h', h'', f)$ be a system as above, where $f:B\to\R$ is a submersion at any point of the zero locus $T:=Z(f)$, which we assume to be compact.  Put
 $$F':=\resto{E'}{T}\ ,\ F'':=\resto{E''}{T}\   .
 $$

Identifying as usual $B$ with the zero section of $E$ we have a natural identification of $S^1$-bundles
$$\resto{T_E}{B}= T_B\oplus  E' \oplus  \bar E'' \ .
$$

Let $U\subset E$ be  an $S^1$-invariant open neighborhood of  $T\subset E$ and $\varphi:U\to \R $
a smooth, $S^1$-invariant map with the following properties:\\

\begin{enumerate}[P1.]
\item  \label{intersection}  $Z(\resto{\varphi}{B})=T$, and $d_x(\resto{\varphi}{B})=d_x f$ for any $x\in T$.
\item   \label{linearization} For  every $x\in  T$ one has 
$$ \resto{D_x \varphi}{E'_x}=0\ ,\  \resto{D_x \varphi}{E''_x}=0\ ,
$$
\item  \label{hessian} For every  $x\in T$ the point $x$ is a non-degenerate critical point of the fiberwise restriction $\varphi_x:=\resto{\varphi}{E_x \cap U}$, and the second derivative at $x$ of this restriction   is
$$D^2_x(\varphi_x)((u'_1,u''_1),(u'_2,u''_2))=\Re (h'_x(u'_1,u'_2))-\Re (h''_x(u''_1,u''_2))\ .
$$ 
\end{enumerate}
\vspace{3mm}

Using conditions P\ref{intersection}, P\ref{linearization}               we obtain isomorphisms of $S^1$-bundles over $T$
$$\resto{T_{Z(\varphi)}}{T}=T_T\oplus F'\oplus\bar F''\ ,\ N_T^{Z(\varphi)}=F'\oplus \bar F''\ ,
$$
hence $\hat Q_f$ and the blowup $S^1$-quotient of $Z(\varphi)$   have the same boundary.
The following proposition shows that, replacing $m^f$ by $\varphi$ in the definition of $\hat Q_f$ one obtains the same {\it complex} manifold with boundary around this common  boundary.

\begin{pr} \label{prop} Let Let  $(B,p':E'\to B, p'':E''\to B, h', h'', f)$ be a system as above, where $f:B\to\R$ is a submersion at any point of $T:=Z(f)$, which we assume to be compact. Let  $\varphi:U\to\R$ be a smooth, $S^1$-invariant map satisfying the conditions P1 - P3 above.
There exists an $S^1$-invariant, open neighborhood $V$ of $T$ in $U$ such that
\begin{enumerate}[1.]
\item $\varphi$ is a submersion at any vanishing  point  of $\resto{\varphi}{V}$, in particular the zero locus $Z_{V,\varphi}:=Z(\resto{\varphi}{V})$ is a smooth real hypersurface of $V$.
\item Denoting $ Z_{V,\varphi}^*:=Z_V^\varphi\setminus B$ one has $Z_{V,\varphi}^*\subset E_f^\st$.

\item  The map
$Q^*_{V,\varphi}:= {Z_{V,\varphi}^*}/{S^1}\to  { E_f^\st}/{\C^*} =Q^*_f
$
induced by the inclusion $Z_{V,\varphi}^*\hookrightarrow E^\st_f$ has the properties:
\begin{enumerate}
\item \label{3a} extends to a smooth open embedding of manifolds with boundary
$$\hat Q_{V,\varphi}:=\qmod{\widehat{Z_{V,\varphi}}}{S^1} \hookrightarrow \qmod{\widehat{Z(m^f)}}{S^1}=\hat Q_f
$$
which induces the identity map between the boundaries (via the identifications $\partial \hat Q_{V,\varphi}=\P(F'\oplus \bar F'')=\partial \hat Q_f$).
\item \label{3b} is a an open embedding  which becomes holomorphic if we endow $Q^*_{V,\varphi}$ with the complex structure induced from  ${ E_f^\st}/{\C^*}$. \end{enumerate}
\end{enumerate}
 
\end{pr}

\begin{proof}
Note first that a sufficiently small  tubular neighborhood  of $T$ in $B$ can be smoothly identified with $(-\varepsilon,\varepsilon)\times T$ such that $f$ is given by the projection on the first factor. Furthermore, using Hermitian connections on $E'$,  $E''$ and parallel transport along the curves $(-\varepsilon,\varepsilon)$, one can identify these Hermitian bundles with $E'=  F'\times (-\varepsilon,\varepsilon)$, $E''=  F''\times (-\varepsilon,\varepsilon)$ where $F':=\resto{E'}{T}$, $F'':=\resto{E''}{T}$. Therefore, since our problem is local with respect to $T$, we can suppose that

\begin{itemize}
\item  $B= T\times (-\varepsilon,\varepsilon)$ (as real differentiable manifolds) and $f: T\times (-\varepsilon,\varepsilon)\to \R$ is given by the projection on the second factor.
\item  $E'=  F'\times(-\varepsilon,\varepsilon)$, $E''=  F''\times (-\varepsilon,\varepsilon)$, where $q':F'\to T$, $q'':F''\to T$  are Hermitian bundles on $T$. 
\end{itemize}
1.  By  P\ref{intersection} the map $\varphi$ is a submersion at any point $x\in T$, hence it is a submersion on an open neighborhood $V$ of $T$. Since $S^1$ is compact we may suppose that $V$ is $S^1$-invariant.\\
\\
2. Put $F:=F'\times_T F''$. Using P\ref{linearization} we see that the hypersurface $Z_{V,\varphi}\subset V$ is vertical along $T$, i.e,  the restriction of its tangent bundle to $T$ coincides with $F$. Therefore,  we may suppose that, for sufficiently small, relatively compact, $S^1$-invariant open neighborhoods $V'$, $V''$ of $T$ in $F'$ and $F''$, the following holds:
\begin{itemize}
\item  $V=(V'\times_T V'')\times (-\varepsilon,\varepsilon)$,
\item  $Z_{V,\varphi}$ is the graph of an $S^1$-invariant function 
$\chi: V'\times_T V''\to (-\varepsilon,\varepsilon)$
 which vanishes on $T$, and whose differential vanishes at any point of $T$.
\end{itemize}
   In other words the first order jet of $\chi$ along $T$ vanishes. The equality $Z_{V,\varphi}=\mathrm{graph}(\chi)$ implies $\varphi(y,\chi(y))=0$ for any $y\in V'\times_T V''$. Differentiating  twice this identity at a point $x\in T$ in fiber directions, and taking into account P\ref{intersection} and P\ref{hessian} we obtain
$$D^2_{x}\chi((u_1',u_1''),(u_2',u_2''))=- D^2_x(\varphi_x)((u'_1,u''_1),(u'_2,u''_2))$$ $$=-\Re (h'_x(u'_1,u'_2))+\Re(h''_x(u''_1,u''_2))\ .
$$  
The order 2-Taylor development of the fiber restriction $\chi_x:=\resto{\chi}{F_x}$  reads
\begin{equation}\label{chi}
\chi_x(u',u'')=-\frac{1}{2}(h'_x(u',u')-h''_x(u'',u''))+r_x(u',u'')\ ,
\end{equation}
where, putting $u=(u',u'')$ the rest  $r_x$ is given by the integral formula
\begin{equation}\label{rest}
r_x(u)=  \int_0^1\frac{(1-t)^2}{2} D^3_{tu}(\chi_{x})(u,u,u)dt\ .
\end{equation}
Since we supposed $V'$, $V''$ to be relatively compact, we get a uniform bound
\begin{equation}\label{Rest}
|r_x(u',u'')|\leq M \{h'_x(u',u')+ h''_x(u'',u'')\}^{\frac{3}{2}}\ .
\end{equation}
Taking $V'$, $V''$ sufficiently small   we will have 
\begin{equation}\label{estimates}
|r_x(u',0)|\leq \frac{h'_x(u',u')}{4}  ,\ |r_x(0,u'')|\leq \frac{h''_x(u'',u'')}{4} ,\ \forall x\in T\ \forall (u',u'')\in V'_x\times V''_x .
\end{equation}
On the other hand, letting $x$ vary in $T$, the hypersurface 
$$Z_{V,\varphi}=\mathrm{graph}(\chi)\subset  (V'\times_T V'')\times (-\varepsilon,\varepsilon)$$
 is defined by the equation
$$t=-\frac{1}{2}(h'(u',u')-h''(u'',u''))+r(u',u'')\ ,
$$
hence, taking into account (\ref{estimates}), we obtain
$$(u',0,t)\in Z_{V,\varphi}^*\Rightarrow t< 0\ ,\ (0,u'',t)\in Z_{V,\varphi}^*\Rightarrow t> 0\ ,$$
hence  $Z_{V,\varphi}^*\subset E^\st_f$ as claimed.
\\ \\
3. We have to compare the quotients $\hat Q_{V,\varphi}=\widehat{Z_{V,\varphi}}/S^1$, $\hat Q_f=\widehat{Z(m^f)}/{S^1}$.  We will first compare the two hypersurfaces 
$$Z_{V,\varphi}= \mathrm{graph}(\chi)\ ,\ Z(m^f)=\mathrm{graph}(\chi_f)\ ,$$
where $\chi:V'\times_TV''$ has been defined above and $\chi_f:F'\times_T F''\to\R$ is given by
$$\chi_f(v',v''):=-\frac{1}{2}(h'_x(v',v')-h''_x(v'',v''))\ .
$$
The maps $g_f:F=F'\times_T F''\to Z(m^f)$, $g:V'\times_T V''\to Z_{V,\varphi}$ given by
$$g_f(v',v''):=(v',v'',\chi_f(v',v''))\ ,\ g(v',v''):=(v',v'',\chi(v',v''))
$$
are $S^1$-equivariant diffeomorphisms, hence they induce $S^1$-equivariant diffeomorphisms
$$\hat g_f:\hat F\to \hat Z(m^f)\ ,\ \hat g:\widehat{V'\times_T V''}\to \hat Z_{V,\varphi}
$$
between the corresponding spherical blowups.
 Note now that the $\C^*$-action on $F$ given by $(\zeta,v',v'')\mapsto (\zeta v',\zeta^{-1} v'')$ extends to a $\C^*$-action on the spherical blowup $\hat F$, which is given by
$$\zeta\cdot (r,w)= \bigg( r ( |\zeta|^2\|w'\|^2+|\zeta|^{-2}\|w''\|^2)^{\frac{1}{2}}, \frac{1}{(|\zeta|^2\|w'\|^2+|\zeta|^{-2}\|w''\|^2)^{\frac{1}{2}}}(\zeta w', \zeta^{-1} w'')\bigg)
$$
and leaves invariant $\partial \hat F$.
Define $U:\widehat{V'\times_T V''}\to \widehat {F'\times_T F''}$   by %
$$U(y)=\rho(y)\cdot y\ ,$$
where $\rho:\widehat{V'\times_T V''}\to \R_{>0}$ is the smooth function given by Lemma \ref{implicit} below. This map acts as the identity on the boundary $\widehat{V'\times_T V''}$ and on its normal line bundle. The composition  $\Ug:=\hat g_f\circ U\circ \hat g^{-1}$ is a smooth, $S^1$-equivariant map $\hat Z_{V,\varphi}\to \hat Z(m^f)$ which induces $\id_{S(F)}$ on the common boundary, and is a local diffeomorphism at any point of this boundary. Therefore, passing to $S^1$-quotients, $\Ug$ induces a smooth map $\ug=\hat Q_{V,\varphi}\to \hat Q_f$ which acts as identity on $\partial \hat Q_{V,\varphi}$ and is a local diffeomorphism at any point of $\partial \hat Q_{V,\varphi}$.  Applying the inverse function theorem at the boundary points and replacing   $V$ by a smaller $S^1$-invariant open neighborhood of $T$ if necessary, $\ug$ will become a smooth open embedding. Moreover, using (\ref{chi-eq}), we see that for a point $y=(v',v'',t)\in Z_{V,\varphi}^*$ one has 
$$\Ug(y)=(\rho(y) v',\rho(y)^{-1} v'',t)\in \C^* y\ ,$$
 which shows that $\ug$ is an extends the natural map $Q^*_{V,\varphi}\to Q^*_f$  induced by the inclusion $Z_{V,\varphi}^*\hookrightarrow E^\st_f$. This proves claim \ref{3a}. Claim \ref{3b} is an obvious consequence of \ref{3a}. 
\end{proof}
\begin{lm}\label{implicit}
Let $V'$, $V''$ be  open,  $S^1$-invariant open neighborhoods of the zero sections in $F'$, $F''$ as in the proof of conclusions 1., 2. of Proposition \ref{prop}.  The equation
\begin{equation}\label{chi-eq}\chi_f(\rho v',\rho^{-1} v'')=\chi(v',v'')
\end{equation}
has a unique solution for every $(v',v'')\in (V'\times_T V'')\setminus T$ and the obtained function $(V'\times_T V'')\setminus T\to (0,\infty)$ extends to a smooth, $S^1$-invariant function %
$$\rho: \widehat{V'\times_T V''}\to (0,\infty)$$
with $\resto{\rho}{S(F'\oplus F'')}\equiv 1$.\end{lm}
\begin{proof}
It's easy to see that (\ref{chi-eq}) has a unique solution in $(0,\infty)$ when $v=(v',v'')$ does not belong to the zero section $T\subset F$. Indeed, for $v=(v',v'')\ne 0$, the equation $\chi_f(\rho v',\rho^{-1} v'')=c$ has always a unique positive solution except when $v''=0$ and $c>0$ and when $v'=0$ and $c<0$. But conclusion 2. of Proposition \ref{prop}  shows that $\chi(v',0)<0$ when $v'\ne 0$ and $\chi(0,v'')>0$ when $v''\ne 0$.

The obtained map $(V'\times_T V'')\setminus T \to (0,\infty)$  is $S^1$-invariant because $\chi_f$ and $\chi$ are $S^1$-invariant. We have to prove that this function extends to a smooth map $\rho:\widehat{V'\times_F V''} \to (0,\infty)$ satisfying $\resto{\rho}{S(F)}\equiv 1$. Applying Lemma \ref{renorm} below to the functions $\alpha$, $\beta:(0,\infty)\times V'\times_T V''\to \R$ given by 
$$\alpha(\rho,v',v'')=\chi_f(\rho v',\rho^{-1} v'')-\chi(v',v'')\ ,\ $$
$$\beta(\rho,v',v'')=\frac{\partial}{\partial \rho}\alpha(\rho,v',v'')=-\rho^{-1}\left( \rho^2 h'(v', v')+\rho^{-2}h''(v'',v'')\right)$$
we obtain two {\it smooth} function $\hat \alpha$, $\hat \beta: (0,\infty)\times\widehat{V'\times_F V''}\to \R$ extending the functions
$$(\rho,v',v'')\mapsto \frac{1}{\|v\|^2} \alpha(\rho,v',v'')\ ,\  (\rho,v',v'')\mapsto\frac{1}{\|v\|^2} \beta(\rho,v',v'')\ ,$$
and whose restrictions to the boundary $(0,\infty)\times S(F)$ are given by
\begin{equation}\label{first}\hat \alpha(\rho,w',w'')=-\frac{\rho^2-1}{2}\left(  h'(w',w')+\frac{1}{\rho^2}  h''(w'',w'')\right)
\end{equation}
\begin{equation}\label{second}
\hat \beta(\rho,w',w'')=-\rho^{-1}\left( \rho^2 h'(w', w')+\rho^{-2}h''(w'',w'')\right)\ .
\end{equation}

For (\ref{first}) we used (\ref{chi}) and (\ref{Rest}) to compute $\lim_{r\to 0} \frac{1}{r^2}\chi(rw)$. 
Since $\frac{\partial}{\partial\rho} \hat \alpha=\hat \beta$ away of the boundary  $(0,\infty)\times S(F)$, it follows that this equality holds on the whole $(0,\infty)\times\widehat{V'\times_F V''}$. 
The first formula shows that on the boundary $(0,\infty)\times S(F)$ the equation $\hat \alpha(\rho,w',w'')=0$ has a unique positive solution $\rho=1$ and the second formula shows that $\frac{\partial}{\partial\rho} \hat \alpha(1,w',w'')\ne 0$ for any $(w',w'')\in S(F)$. Using the implicit function theorem we see that the equation $\hat \alpha(\rho,w',w'')=0$ defines a smooth positive function on a neighborhood of the boundary $S(F)$ in $\widehat{V'\times_F V''}$.
\end{proof}

\begin{lm}\label{renorm}
Let $F\to X$ be a real vector bundle of finite rank $m$, $W\subset F$ an open neighborhood of the zero section of $F$, and $\hat W$ the spherical blowup of $W$ along the zero section $X\subset F$. Let $\tau:  W\to \R$ be a smooth map whose fiber restrictions $\tau_x:=\resto{\tau}{W_x}$ satisfy the inequality
\begin{equation}\label{estimate} | \tau_x(v)|\leq C_x\|v\|^k \ \forall x\in X\ \forall v\in W_x 
\end{equation}
for a continuous function $C:X\to (0,\infty)$ and a positive integer $k$.  Then   the map 
$W\setminus X=\mathrm{Int}(\hat W)\to \R$
 given by 
$$v\mapsto  \frac{1}{\|v\|^k} \tau (v)$$
extends to a smooth map $\hat\tau:\hat W\to \R$ whose restriction to the boundary $\partial \hat W=S(F)$ is given  by
\begin{equation}\label{r=0}\hat\tau(w)=\lim_{r\to 0} \frac{1}{r^k} \tau(rw)=\frac{1}{k!} D^k_{0_x}\tau_x(\underbrace{w,w\dots w}_{k})\ \forall w\in S(F_x)\ .
\end{equation}
\end{lm}

\begin{proof} Choose $N>k$. The order $N$-Taylor expansion at $0_x$ of the fiber restriction $\tau_x$   reads

$$\tau_x(v)=\sum_{0\leq l\leq N} \frac{1}{l!} D^l_{0_x}\tau_x(\underbrace{v \dots v}_{l}) +    \frac{1}{N!} \int_0^1 (1-t)^N D^{N+1}_{tv}\tau_x(\underbrace{v \dots v}_{N+1})dt\ .
$$
The assumption (\ref{estimate}) implies
$$D^l_{0_x}\tau_x=0\ \forall x\in X\ \forall l\in\{0,\dots,k-1\}\ .
$$
Putting $v=rw$ with $\|w\|=1$ we get for $r>0$
$$\frac{1}{\|v\|^k} \tau (v)=\sum_{k\leq l\leq N} \frac{1}{l!} D^l_{0_x}\tau_x(\underbrace{w \dots w}_{l})  r^{l-k}+   \frac{r^{N-k}}{N!} \int_0^1 (1-t)^N D^{N+1}_{trw}\tau_x(\underbrace{w \dots w}_{N+1})dt\ ,
$$
which can obviously be smoothly extended to a smooth function $\hat\tau$ on spherical blow up $\hat W$ whose restriction to the boundary $\partial \hat W=S(F)$ is given by (\ref{r=0}). \end{proof}

\begin{re}\label{MetOnT}
Proposition \ref{prop} shows that, ``around its boundary", the blowup flip passage $\hat Q_f$ depends only on 
\begin{enumerate}
\item the trivialization of the normal line bundle $N_T$ induced by $d f$,
\item the restrictions $\hg'$, $\hg''$ of the metrics $h'$, $h''$  on $F':=\resto{E'}{T}$, $F'':=\resto{E''}{T}$.
\end{enumerate}  
\end{re}

 \section{The holomorphic model theorem}
 
 \subsection{The blowup flip passage associated with a circle of regular reductions}\label{a-flip-passage}
 
 Let $(X,g)$ be a Gauduchon surface with $p_g(X)=0$ and $b_1(X)=1$, and let $D$, $L$ be Hermitian line bundles on $X$. We fix Hermite-Einstein connections $a\in {\cal A}(D)$, $b_0\in{\cal A}(L)$   such that
$$\int_X i\Lambda_g F_{b_0} =\frac{1}{2}\int_X i\Lambda_g F_{a}\ ,
$$
and we denote by $\delta:=\bar\partial_a$, $\sigma_0:=\bar\partial_{b_0}$ the corresponding integrable semiconnections.

Put $c:=c_1(L)$. In section \ref{GaugePic} we identified the component  $\Pic^c(X)$ of the Picard group of $X$ with the moduli space ${\cal M}(L)$, which has a very simple description as a finite dimensional quotient:
 $${\cal M}(L)=\qmod{\Sigma}{G_{x_0}}\ ,
 $$
for a base point $x_0\in X$. This identification is given explicitly by $[\sigma]\mapsto [{\cal L}_\sigma]$. In our case  $\Sigma$ is an affine complex line, and $G_{x_0}$ is a cyclic group canonically isomorphic to $2\pi H^1(X,i\Z)$.  

For $\sigma\in\Sigma$ we denote by $b_\sigma$ the Chern connection of the  Hermitian line bundle  ${\cal L}_\sigma$, which will also be   Hermite-Einstein (see Remark \ref{FinDImKH} in the Appendix).   
Put
 $$\check{L}:=L^\smvee\otimes D \ , \ L':=\check{L}^\smvee\otimes L\ ,\ L'':=L^\smvee\otimes \check{L}\ .
$$
Similarly, for a connection $b$, a semiconnection $\sigma$ and a holomorphic structure ${\cal L}$ on $L$ we put
$$\check{b}:=b^\vee\otimes a\ ,\ b':=\check{b}^\smvee\otimes b\ ,\ b'':= b^\vee\otimes\check{b}\ ,\ \check{\sigma}:=\sigma^\smvee \otimes \delta\ ,\ \sigma':=\check\sigma^\smvee\otimes\sigma \ ,\ \sigma'':=\sigma^\smvee\otimes\check\sigma\ ,$$
$$\check{{\cal L}}:={\cal L}^\vee\otimes {\cal D}\ ,\ {\cal L}':=\check{\cal L}^\vee\otimes {\cal L}\ ,\ {\cal L}'':={\cal L}^\vee\otimes \check{{\cal L}}\ .
$$
Consider the circle 
$$T:=\bigg\{[{\cal L}]\in\Pic^c(X)|\ \deg_g({\cal L})
=\frac{1}{2}\deg_g({\cal D}) \bigg\}\ .
$$
Our identification  ${\cal M}(L)=\Pic^c(X)$  restricts to an identification 
$$\qmod{\Sigma_0}{G_{x_0}}\textmap{\simeq} T\ ,$$ 
 where $\Sigma_0:=\sigma_0+ H^{0,1}_{\rm cl}$. For $\sigma\in \Sigma_0$ put
$$v'_\sigma :=\Lambda_g\partial_{b'_\sigma}: A^{0,1}(L')\to A^0(L')\ ,\  v''_\sigma =\Lambda_g\partial_{b''_\sigma}: A^{0,1}(L'')\to A^0(L'')$$
$$\Hg'_\sigma:=\ker(\sigma': A^{0,1}(L')\to A^{0,2}(L'))\cap \ker(v'_\sigma)\ , $$  $$\Hg''_\sigma:=\ker(\sigma'': A^{0,1}(L'')\to A^{0,2}(L''))\cap \ker(v''_\sigma)\ .
$$
 Using Proposition \ref{H1} proved in the Appendix we obtain
\begin{lm} \label{H1iso} Suppose $2c\ne c_1(D)$ and let $\sigma\in \sigma_0+ H^{0,1}_{\rm cl}$. Then
\begin{enumerate}[1.]
\item $h^0({\cal L}'_\sigma)=h^0({\cal L}''_\sigma)=0$,
\item The natural morphisms 
$\Hg'_\sigma\to H^1({\cal L}'_\sigma)$, $\Hg''_\sigma\to H^1({\cal L}''_\sigma)$
 are isomorphisms.
\end{enumerate}
\end{lm}
\begin{proof}
Since $\deg_g({\cal L}'_\sigma)=\deg({\cal L}''_\sigma)=0$,  the Einstein constants  of $b'_\sigma$, $b''_\sigma$  vanish. Therefore Proposition \ref{H1} applies, and shows that any holomorphic section of ${\cal L}'_\sigma$ (${\cal L}''_\sigma$) is $b'_\sigma$-parallel ($b''_\sigma$-parallel). But a non-trivial parallel section of ${\cal L}'_\sigma$ (${\cal L}''_\sigma$) would define a bundle isomorphism ${\cal L}^{\otimes 2}\to {\cal D}$, which contradicts the hypothesis. The second statement follows directly from Proposition \ref{H1}. \end{proof}

Let now $\mathscr{L}=\mathscr{L}_{x_0}$ the Poincaré line bundle (associated with the base point $x_0$) on ${\cal M}(L)\times X$ and denote by $p_1: {\cal M}(L)\times X\to {\cal M}(L)$, $p_2: {\cal M}(L)\times X\to X$ the two projections. Consider the coherent sheaves
$${\cal H}':=R^1(p_1)^*(\mathscr{L}^{\otimes 2}\otimes p_2^*({\cal D}^\vee))\ ,\ {\cal H}'':=R^1(p_1)^*(\mathscr{L}^{\otimes -2}\otimes p_2^*({\cal D}))
$$
on ${\cal M}(L)$.
\begin{dt}\label{regular}
A pair $(L,{\cal D})$  as above  with $2c_1(L)\ne c_1({\cal D})$will be called regular  if $h^2({\cal L}')=h^2({\cal L}'')=0$ for any $[{\cal L}]\in T$.
\end{dt}

The regularity of the pair $(L,{\cal D})$   is equivalent to the condition ``${\cal T}_\lambda$ is a circle of regular reductions" mentioned in the introduction (see Corollary 1.21 in \cite{Te3}).
Using Lemma \ref{H1iso}, the Riemann-Roch theorem and Grauert's semicontinuity,  local triviality and base change theorems, we obtain
\begin{pr}
Let $(L,{\cal D})$ be a regular pair. For any sufficiently small $\varepsilon>0$ the restrictions of ${\cal H}'$, ${\cal H}''$ to the annulus
$${\cal M}(L)_\varepsilon:=\left\{[{\cal L}]\in \Pic^c(X)|\ \pi\big|\deg_g({\cal L})
=\frac{1}{2}\deg_g({\cal D})\big|<\varepsilon\right\}
$$
are locally free of ranks
$$r'=-\frac{1}{2} (2c -c_1(E))(2c -c_1(E)+c_1(X))\ ,\ r''=-\frac{1}{2} (-2c +c_1(E))(-2c +c_1(E)+c_1(X))
$$
respectively, and for any $l\in {\cal M}(L)_\varepsilon$ one has canonical identifications 
$${\cal H}'(l)=H^1(\mathscr{L}'_l) \ ,\ {\cal H}''(l)=H^1(\mathscr{L}'')\ .$$
\end{pr}
\begin{re} Since the annulus ${\cal M}(L)_\varepsilon$ is  
Stein and homotopically equivalent to a circle, it follows that the bundles $\resto{{\cal H}'}{{\cal M}(L)_\varepsilon}$, $\resto{{\cal H}''}{{\cal M}(L)_\varepsilon}$ are in fact trivial \cite{Gr}.
\end{re}
 In section \ref{GaugePic} we showed that the Poincaré line bundle $\mathscr{L}$ comes with a canonical Hermitian metric  and is fiberwise  Hermite-Einstein in the $X$-directions.  Therefore $\mathscr{L}_l'$, $\mathscr{L}_l''$ become Hermitian line bundles. Using  the isomorphisms given by the second conclusion of Lemma \ref{H1iso} and the $L^2$-inner product on the spaces $\Hg'_\sigma$, $\Hg''_\sigma$ we get Hermitian metrics $\hg'$, $\hg''$ on the bundles $\resto{{\cal H}'}{T}$, $\resto{{\cal H}''}{T}$. 

Define $f_{\cal D}:{\cal M}(L)\to \R$ by $f_{\cal D}([\sigma]):=2\pi(\deg({\cal L}_\sigma)-\frac{1}{2}\deg({\cal D}))$. According to Remark \ref{MetOnT} the system
$$\left({\cal M}(L)_\varepsilon,\resto{{\cal H}'}{{\cal M}(L)_\varepsilon}, \resto{{\cal H}''}{{\cal M}(L)_\varepsilon}, \hg', \hg'', \resto{f_{\cal D}}{{\cal M}(L)_\varepsilon}\right)
$$
can be used to define a  blowup  flip passage (around its boundary) in a coherent way. We will denote by $\hat Q$ this blowup flip passage. Recall that its interior comes with a complex structure, and its boundary can be identified with  the projective bundle $\P(\resto{{\cal H}'}{T}
\oplus  \overline{\resto{{\cal H}''}{T}} )$ over $T$.

 \subsection{The moduli space ${\cal N}$}\label{interm}
 
The family of operators
\begin{equation}\label{D1} A^0(L')\textmap{\sigma'}  A^{0,1}(L')\ ,\ A^0(L'')\textmap{\sigma''}  A^{0,1}(L'')
\end{equation}
\begin{equation}\label{D2} A^{0,1}(L')\textmap{\sigma'} A^{0,2}(L')\ ,\ A^{0,1}(L'')\textmap{\sigma''} A^{0,2}(L'')
\end{equation}
associated with points $\sigma\in\Sigma$ are $G^\C$-equivariant, if we let the group $G^\C$ act on $\Sigma$ by %
$$\varphi\cdot \sigma:=\varphi\circ\sigma\circ \varphi^{-1}=\sigma-\varphi^{-1}\bar\partial \varphi$$
 and on the spaces $A^{0,q}(L')$,  $A^{0,q}(L'')$ by 
$$\varphi\cdot \alpha'=\varphi^2 \alpha'\ ,\ \varphi\cdot \alpha''=\varphi^{-2} \alpha'\ .
$$
Let $\Sigma_\varepsilon$ the preimage of the annulus ${\cal M}(L)_\varepsilon$  under the   quotient map $\Sigma\to {\cal M}(L)$. Therefore $\Sigma_\varepsilon$ is a symmetric neighborhood of the real line $\Sigma_0$ in $\Sigma$.  
Suppose now that the pair $(L,{\cal D})$ is regular in the sense of Definition \ref{regular}.  By Lemma \ref{H1iso} and Grauert's semicontinuity theorem it follows for sufficiently small $\varepsilon>0$ the following holds:  for any $\sigma\in \Sigma_\varepsilon$ the two operators if  (\ref{D1}) are injective, and the two operators in (\ref{D2}) are surjective. From now on we suppose that $\varepsilon$ is sufficiently small such that these properties hold on $\Sigma_\varepsilon$.  We will need two  {\it holomorphic}, $G^\C$-equivariant families  of operators $v'_\sigma$, $v''_\sigma$ defined for every $\sigma\in \Sigma_\varepsilon$ such that  $\ker(v'_\sigma)$, $\ker(v''_\sigma)$  are complements of the images of the two operators in (\ref{D1}) for every $\sigma\in \Sigma_\varepsilon$. The families $\sigma\mapsto \Lambda_g\partial_{b'_\sigma}$, $\sigma\mapsto \Lambda_g\partial_{b''_\sigma}$ are $G$-equivariant, but unfortunately they are not holomorphic. They are antiholomorphic. Since we are interested in holomorphic models, this complicates our arguments.
\begin{pr} \label{adjoints} For sufficiently small $\varepsilon>0$ there exists  $G^\C$-equivariant families of operators
$$v'_\sigma:A^{0,1}(L')\to A^0(L')\ ,\  v''_\sigma: A^{0,1}(L'')\to A^0(L'') \ , 
$$
$$w'_\sigma: A^{0,2}(L')\to A^{0,1}(L')\ ,\ w''_\sigma: A^{0,2}(L'')\to A^{0,1}(L'')  
$$
depending holomorphically on $\sigma\in\Sigma_\varepsilon$ such that for any  $\sigma\in\Sigma_\varepsilon$ it holds
\begin{enumerate} \item \label{compl} $\ker(v'_\sigma)$, $\ker(v''_\sigma)$ are topological complements of the images of the two operators of $(\ref{D1})$, 
\item \label{iden1} $v'_\sigma\circ w'_\sigma=0$, $v''_\sigma\circ w''_\sigma=0$,
 \item \label{iden2} $w'_\sigma$, $w''_\sigma$ are right inverses of the two operators in (\ref{D2}).
 
\end{enumerate}
In particular, $\im(w'_\sigma)$, $\im(w''_\sigma)$ are topological complements of 
$$\Hg'_\sigma:=\ker(\sigma': A^{0,1}(L')\to A^{0,2}(L'))\cap \ker(v'_\sigma)\ , $$  $$\Hg''_\sigma:=\ker(\sigma'': A^{0,1}(L'')\to A^{0,2}(L''))\cap \ker(v''_\sigma)
$$
in $\ker(v'_\sigma)$, $\ker(v''_\sigma)$ respectively.

\end{pr}
\begin{proof} Taking suitable Sobolev completions, the maps 
$$\Sigma_0\ni \sigma\mapsto v'_\sigma:=\Lambda_g\partial_{b'_\sigma}\ ,\ \Sigma_0\ni \sigma\mapsto v''_\sigma:=\Lambda_g\partial_{b''_\sigma}$$
 become real analytic, $G^\C$-equivariant and take values in a complex Banach space of bounded operators. Therefore these maps extend  holomorphically on   $\Sigma_\varepsilon$ for sufficiently small $\varepsilon>0$, and, by the identity theorem for holomorphic applications,  the extension will still be $G^\C$-equivariant. Since $T=\Sigma_0/G_{x_0}$ is compact,  condition (\ref{compl}) (which is is open with respect to $\sigma$) will hold on a sufficiently small $\Sigma_\varepsilon$.   Similarly, for $\sigma\in \Sigma_0$ let 
$$w'_\sigma: A^{0,2}(L')\to A^{0,1}(L')\ ,\ w''_\sigma: A^{0,2}(L'')\to A^{0,1}(L'')$$
 be the right inverses of the two operators in (\ref{D2}) which take values in the $L^2$-orthogonal complements of $\Hg'_\sigma$ in $K'_\sigma:=\ker(v'_\sigma)$, and $\Hg''_\sigma$ in  $K''_\sigma:=\ker(v''_\sigma)$ respectively. These maps are again real analytic on  $\Sigma_0$ hence, for sufficiently small $\varepsilon>0$, they admit  holomorphic extensions $\Sigma_\varepsilon\ni \sigma \mapsto w'_\sigma$, $\Sigma_\varepsilon\ni \sigma \mapsto w''_\sigma$. The equivariance properties and  the claims (\ref{iden1}), (\ref{iden2}) follow using again the identity theorem for holomorphic maps.
 \end{proof}
 
 Note that, for $\sigma\not\in\Sigma_0$  we cannot expect the operators $w'_\sigma$, $w''_\sigma$ to  take values in the orthogonal complements of $\Hg'_\sigma$, $\Hg''_\sigma$.\\

Choose $\varepsilon>0$ for which the claims of Proposition \ref{adjoints} hold, and endow the product 
$${\cal C}:=V^{0,1}\times \Sigma_\varepsilon\times A^{0,1}(L')\times A^{0,1}(L'')$$
 (see Theorem \ref{Hodge} in the Appendix) with the natural $G^\C$-action which is trivial on the first factor, and acts as explained above on the other factors. The system
$$  \label{FirstSyst}
\left\{\begin{array}{ccc}
v'_\sigma \alpha'&=&0\\ 
v''_\sigma \alpha''&=&0\\ 
\bar\partial v+\alpha'\wedge  \alpha''&=&0\\
\sigma'\alpha '+ 2v\wedge \alpha'&=&0\\
\sigma''\alpha ''- 2v\wedge \alpha''&=&0
\end{array}
\right.    \ ,
\eqno{(\Ng)}
$$
on ${\cal C}$ is $G^\C$-equivariant.  Our hypothesis $p_g(X)=0$, implies that the restriction %
$$\bar\partial_{0}:=\resto{\bar\partial}{V^{0,1}}:V^{0,1}\to A^{0,2}(X)$$
 is an isomorphism, hence the third equation of $(\Ng)$ is equivalent to the quadratic equation $v=-\bar\partial_{0}^{-1}( \alpha'\wedge \alpha'')$. The space of  solutions ${\cal C}^\Ng$   of $(\Ng)$ is a finite dimensional complex space. A solution with trivial $\alpha'$, $\alpha''$-components also has trivial $v$-component and, under our regularity assumption, the map $\Ng$ defined by the left hand terms of $(\Ng)$ is submersive at at any such solution. Therefore ${\cal C}^\Ng$ is smooth at such a point with tangent space
\begin{equation}\label{TCNg}
T_{(0,\sigma,0,)} {\cal C}^\Ng=H^{0,1}\oplus \Hg'_\sigma\oplus \Hg''_\sigma=H^{0,1}\oplus \Hg_\sigma\ ,
\end{equation}
where we have put $\Hg_\sigma:=\Hg'_\sigma\oplus \Hg''_\sigma$. The quotient
$${\cal N}:=\qmod{{\cal C}^\Ng}{G_{x_0}}\ .
$$
by the cyclic group $G_{x_0}$ comes with a residual $\C^*$-action given explicitly by $\zeta\cdot(v,\sigma,\alpha',\alpha''):=(v,\sigma,\zeta^2\alpha',\zeta^{-2}\alpha'')$. The fixed point locus ${\cal N}^{\C^*}$ is the space of orbits  of points with vanishing $(\alpha',\alpha'')$-component, hence  it can be identified with $\Sigma_\varepsilon/G^\C={\cal M}(L)_\varepsilon$, and ${\cal N}$ is smooth around the fixed point locus.  

The assignments $\sigma\mapsto \Hg'_\sigma$, $\sigma\mapsto \Hg''_\sigma$ are holomorphic and $G_{x_0}$-equivariant, and  the canonical maps $\Hg'_\sigma\to H^1({\cal L}'_{\sigma})$, $\Hg''_\sigma\to H^1({\cal L}''_{\sigma})$ are isomorphisms. Factorizing by $G_{x_0}$ we obtain holomorphic bundles $\Hg'$, $\Hg''$ on ${\cal M}(L)_\varepsilon$ with obvious isomorphisms $\Hg'=\resto{{\cal H}'}{{\cal M}(L)_\varepsilon}$,
$\Hg''=\resto{{\cal H}''}{{\cal M}(L)_\varepsilon}$.  Put $\Hg:=\Hg'\oplus \Hg''$. Via the identification ${\cal N}^{\C^*}={\cal M}(L)_\varepsilon$, the restriction of the holomorphic tangent bundle  ${\cal T}_{\cal N}$   to ${\cal N}^{\C^*}$ is
\begin{equation}\label{TN}
\resto{{\cal T}_{\cal N}}{{\cal N}^{\C^*}}={\cal T}_{{\cal M}(L)}\oplus \Hg=\resto{{\cal T}_{\Hg} }{{\cal M}(L)_\varepsilon}\ .
\end{equation}

The following remark (whose proof will be omitted) explains the role of the moduli space ${\cal N}$ in our arguments: its $\C^*$-quotient  
is mapped naturally to the quotient ${\cal A}^{0,1}_\delta(E)^{\rm int}$.
\begin{re}  \label{eta} Put $E=L\oplus \check{L}$. The map 
$ {\eta}:  {\cal C} \to  {\cal A}^{0,1}_\delta(E)$
given by 
$$ {\eta}(v,\sigma,\alpha',\alpha''):=\left(
\begin{matrix}
\sigma+v&\alpha'\\
\alpha''& \check{\sigma} -v
\end{matrix}
 \right)\ .
$$
has the properties:
\begin{enumerate}[1.]
\item is holomorphic and equivariant with respect to the group monomorphism 
$$\iota:G^\C\to {\cal G}_{E}^\C:=\Gamma(X,\SL(E))\ ,\ \iota(\varphi):= \left(
\begin{matrix}\varphi&0\\ 0& \varphi^{-1}
\end{matrix}
 \right)\ ,
$$
\item  maps ${\cal C}^\Ng$ into ${\cal A}^{0,1}_\delta(E)^{\rm int}$, hence it induces a map 
$$[\eta]:{\cal N}/\C^*\to {\cal A}^{0,1}_\delta(E)^{\rm int}/{\cal G}_{E}^\C\ ,$$
\item There exists an open neighborhood $W$ of $T$ in ${\cal N}/\C^*$ such that the restriction 
$$[\eta^*]:=\resto{[\eta]}{W\setminus {\cal N}^{\C^*}}$$
of $[\eta]$ takes values in ${\cal M}^\si_{\cal D}(E)$ and is holomorphic.

\end{enumerate}

\end{re}

 One can prove that, if $W$ is sufficiently small,  the restriction $[\eta^*]$ is an open embedding into ${\cal M}^\si_{\cal D}(E)$ (hence in particular injective). Since we are interested only in moduli space polystable structures, we will not need this result. Note that $W\setminus {\cal N}^{\C^*}$ is not Hausdorff in general.
 \\

${\cal N}$ comes with a natural holomorphic map 
$$\pi:{\cal N} {\to} {\cal M}(L)_\varepsilon,\ [v,\sigma,\alpha',\alpha'']\mapsto [\sigma]\ .$$
 More precisely $\pi$ is a holomorphic (non-linear) subfibration of the vector bundle  $V^{0,1}\times (A'\times_{{\cal M}(L)_\varepsilon} A'')$ over ${\cal M}(L)_\varepsilon$, where $A'$, $A''$ are the bundles associated with the principal $G_{x_0}$-bundle $p_\varepsilon:\Sigma_\varepsilon\to {\cal M}(L)_\varepsilon$ and the natural representations of $G_{x_0}$ in  $A^{0,1}(L')$, $A^{0,1}(L'')$ respectively. After suitable Sobolev completions $A'$, $A''$ become holomorphic Banach bundles, and ${\cal N}$ is tangent to the finite rank subbundle $\{0\}\times\Hg$ along the zero section.  Is important to note that, using the operator families $(w'_\sigma)$, $(w''_{\sigma})$, the fibration $\pi$ can be {\it holomorphically} ``linearized" over  ${\cal M}(L)_\varepsilon$ in an explicit way:

 \begin{re} \label{isoiso}   The map  
$${\cal C}\to \Sigma_\varepsilon\times A^{0,1}(L')\oplus A^{0,1}(L'')\ ,\ (v,\sigma,\alpha',\alpha'')\mapsto \left(
\begin{matrix}
\sigma\\
 \alpha'+2w'_\sigma ( v\wedge \alpha' )\\
  \alpha''-2w''_\sigma ( v \wedge \alpha'') 
\end{matrix}\right)
 $$
is $G^\C$-equivariant,  induces  a   holomorphic map $\Ag:{\cal C}^\Ng\to p_\varepsilon^*(\Hg)$ over $\Sigma_\varepsilon$ with the properties:
\begin{enumerate}[1.]
\item For any  $\sigma\in \Sigma_\varepsilon$ one has $D_{(0,\sigma,0)}\Ag=\id$, in particular $\Ag$  is a local biholomorphism at any point of ${\cal C}^{\C^*}$.

\item $\Ag$ induces $\C^*$-equivariant holomorphic map $\ag:{\cal N}\to \Hg$ over ${\cal M}(L)$, which is a local biholomorphism at any point of ${\cal N}^{\C^*}$.

\item There exists an open, $S^1$-invariant  neighborhood ${\cal W}$ of $T$ in ${\cal N}$ such that  
$$\ag_{\cal W}:=\resto{\ag}{{\cal W}}:{\cal W}\to \ag({\cal W}) $$
 is a biholomorphism  over ${\cal M}(L)$.

 \item $\ag_{\cal W}$ and its inverse $\bg_{\cal W}:=\ag_{\cal W}^{-1}$  admit $G^\C$-invariant lifts $\Ag_{\cal W}$, $\Bg_{\cal W}$ to the pre-images of their domains in ${\cal C}^\Ng$, $p_\varepsilon^*(\Hg)$ respectively.
\end{enumerate}
 \end{re}
Denoting by $\Hg^*$ the complement of the zero section in $\Hg$ note that 
 \begin{re}\label{etale} The holomorphic map $\{\ag^*\}:\{{\cal W}\setminus{\cal N}^{\C^*}\}/\C^*\to \Hg^*/\C^*$ induced by $\ag$ is is étale.  
  \end{re}
  Note that we cannot expect this map to be injective.

 \subsection{The   moduli spaces ${\cal N}^{\Sg\Ig}$, ${\cal N}^{\Sg\Ig\Mg}$} 
 
To complete the proof of the holomorphic model theorem we need a finite dimensional description of the moduli spaces ${\cal M}^\ASD_a(E)$, $\hat {\cal M}^\ASD_a(E)_\lambda$ around ${\cal T}_\lambda$ (respectively ${\cal P}_\lambda$).  On the product ${\cal A}(L)\times A^1(L')$ consider the equations

$$\left\{\begin{array}{ccc}
\Lambda_g (\partial_{b'} \beta^{01} - \bar\partial_{b'}\beta^{10}) &=&0\\
 \Lambda_g (\partial_{b'}\beta^{01} +  \bar\partial_{b'}{\beta}^{10}) &=&0\\
p_r\{\Lambda_g (F_b -\frac{1}{2} F_a)-i(|\beta^{01}|^2-|\beta^{10}|^2)\}&=&0\\
p_r \Lambda_g d^c (b-b_0)&=&0\\
\end{array}
\right. \eqno{(\Sg)}
$$

$$\left\{\begin{array}{ccc}
 \bar\partial_{b'}{\beta^{01}}&=&0\\
\partial_{b'}{\beta}^{10}&=&0\\
F_b^{0,2}-{\beta^{01}}\wedge  {\beta^*}^{01}&=&0 
\end{array}\right.\eqno{(\Ig)}
$$
$$ \frac{1}{2}\int_X \{i\Lambda_g \big(F_b -\frac{1}{2} F_a\big)+(|\beta'|^2-|\beta''|^2)\}\vol_g=0   \ . \eqno{(\Mg)}
$$
We denote by 
$$\Sg: {\cal A}(L)\times A^1(L')\to A^0(L') \oplus A^0(L') \oplus A^0(X,i\R)_{r}\oplus A^0(X,i\R)_{r}\ ,$$
$$\Ig: {\cal A}(L)\times A^1(L')\to A^{0,2}(L')\oplus A^{0,2}(L'')\oplus A^{0,2}(X)\ ,\ \Mg:{\cal A}(L)\times A^1(L')\to\R $$
the maps defined by the left hand of $(\Sg)$, $(\Ig)$, $(\Mg)$ respectively. Note that $\Sg$, $\Ig$ are $G$-equivariant and $\Mg$ is $G$-invariant. We denote by $\{{\cal A}(L)\times A^1(L')\}^\Sg$, $\{{\cal A}(L)\times A^1(L')\}^{\Sg\Ig}$, $\{{\cal A}(L)\times A^1(L')\}^{\Sg\Ig\Mg}$ the spaces of solutions of the   system indicated as exponent, and by 
${\cal N}^{\Sg\Ig}$, ${\cal N}^{\Sg\Ig\Mg}$  the $G_{x_0}$-quotients of the latter two spaces (which are finite dimensional).
 One has obvious identifications
 $$\{{\cal N}^{\Sg\Ig}\}^{S^1}={\cal M}^\HE(L)\ ,\ \{{\cal N}^{\Sg\Ig\Mg}\}^{S^1}=T\ ,
 $$
where  
$$T:=\qmod{b_0+H_{\rm cl}}{G}=\left\{[b]\in {\cal M}^\HE(L)|\ \int_X(i\Lambda_g F_{b})\vol_g=\frac{1}{2}\int_X (i\Lambda_g F_a)\vol_g\right\}\ ,
$$
corresponds via the Kobayashi-Hitchin identification ${\cal M}^\HE(L)={\cal M}(L)$ to the circle denoted in the previous section by the same symbol. Note that (under our regularity condition)  ${\cal N}^{\Sg\Ig}$ and ${\cal N}^{\Sg\Ig\mg}$   are smooth at any point of $T$. The restriction to $T$ of the corresponding tangent bundles  are:
\begin{equation}
\resto{T_{{\cal N}^{\Ag\Ig}} }{T}=T_{{\cal M}^\HE(L)}\oplus (\bar\Hg''_T\oplus \Hg'_T)\ ,\ \resto{T_{{\cal N}^{\Ag\Ig}} }{T}=T_{T}\oplus(\bar\Hg''_T\oplus \Hg'_T) ,
\end{equation}
where   $\Hg'_T$, $\Hg''_T$ are the restriction to   $T$  of the bundles  $\Hg'$, $\Hg''$ defined in section \ref{interm}, and  are obtained from the families of vector spaces
$$b_0+ H^1_{\rm cl}\ni b\to  \Hg'_b:=\ker \bar\partial_{b'}\cap\ker(\Lambda_g\partial_{b'})\ ,\ b_0+ H^1_{\rm cl}\ni b\to \Hg''_{b}:= \bar\partial_{b''}\cap\ker(\Lambda_g\partial_{b''}).
$$
The bundle $\bar\Hg''_T$ is obtained using the vector spaces  $\bar \Hg''_{b}:=\{(\beta'')^*|\ \beta\in \Hg''_b\}$.
 Put
$$\{{\cal N}^{\Sg\Ig\Mg}\}^{*}:={\cal N}^{\Sg\Ig\Mg}\setminus \{{\cal N}^{\Sg\Ig\Mg}\}^{S^1}\ ,\ {\cal N}^{\Sg\Ig\Mg}_\epsilon:=\{[b,\beta]\in {\cal N}^{\Sg\Ig\Mg}|\ \|\beta\|_{L^\infty} <\epsilon\},$$
$$\{{\cal N}^{\Sg\Ig\Mg}_\epsilon\}^{*}:={\cal N}^{\Sg\Ig\Mg}_\epsilon\setminus  \{{\cal N}^{\Sg\Ig\Mg}\}^{S^1}\ .
$$  

\begin{pr}  \label{theta} Put $E:=L\oplus\check{L}$. The map $\theta:{\cal A}(L)\times A^1(L')\to  {\cal A}_a(E)$ given by
$$(b,\beta)\mapsto \left(\begin{matrix}
d_b & \beta\\ -\beta^* & d_{\check{b}}
\end{matrix}
 \right)
$$
has the following properties:
\begin{enumerate}
\item is an affine isomorphism, which
 is equivariant with respect to the group morphism $G\to {\cal G}_E$ given by $\varphi\mapsto \left(\begin{matrix}
\varphi & 0\\ 0 &  \varphi^{-1}
\end{matrix}
 \right)$,
\item maps ${\cal A}(L)\times A^1(L')^{\Sg\Ig\Mg}$ into ${\cal A}_a^\ASD(E)$ and, for sufficiently small $\epsilon>0$, the induced map 
 $$[\theta]: {\cal N}^{\Sg\Ig\Mg}/S^1\to {\cal M}_a^\ASD(E)$$
 maps  isomorphically  ${\cal N}^{\Sg\Ig\Mg}_\epsilon/S^1$ onto an open neighborhood  ${\cal O}_\lambda$ of ${\cal T}_\lambda$ in  ${\cal M}_a^\ASD(E)$,
and restricts to a diffeomorphism 
$$[\theta^*]:\{{\cal N}^{\Sg\Ig\Mg}_\epsilon\}^*/S^1\to ({\cal O}_\lambda\setminus {\cal T}_\lambda)\subset {\cal M}^\ASD_a(E)^*\ .$$

\end{enumerate}

\end{pr}

We will omit the proof of this statement. We mention only that the pull back of the ASD equation via $\theta$ obviously coincides  with the system formed by $(\Jg)$, $(\Mg)$, the second and the third equation of $(\Sg)$. The role of the first and fourth equation of $(\Sg)$  is to reduce (around $\{b_0+H_{\rm cl}\}\times\{0\}$) the ${\cal G}_E$-factorization involved in the definition of  ${\cal M}_a^\ASD(E)$ to the  $G$-factorization.
\qed

Taking into account this result we can {\it define} the blowup moduli space $\hat {\cal M}_a^\ASD(E)_\lambda$ using the blowup $S^1$-quotient of ${\cal N}^{\Sg\Ig\Mg}_\epsilon$. More precisely, we put
$$\hat {\cal M}_a^\ASD(E)_\lambda:=\{{\cal M}_a^\ASD(E)\setminus {\cal O}_\lambda\}{\textstyle \coprod_{[\theta^*]}}\  \widehat{ {\cal N}^{\Sg\Ig\Mg}_\epsilon}/S^1\ .
$$
One can prove that this construction is equivalent to  the one given in \cite{Te3} section 1.4.2. With this definition we have obviously

\begin{re} \label{blowup-embed} The restriction $[\theta^*]:\{{\cal N}^{\Sg\Ig\Mg}_\epsilon\}^*/S^1\to {\cal M}^\ASD_a(E)^*$ of $[\theta]$ extends continuously to a smooth open   embedding
$\widehat{[\theta]}:  \left\{\widehat{ {\cal N}^{\Sg\Ig\Mg}_\epsilon}/S^1\right\}\to \hat {\cal M}_a^\ASD(E)_\lambda$
which identifies  $\partial \widehat{ {\cal N}^{\Sg\Ig\Mg}_\epsilon}$ with the boundary  ${\cal P}_\lambda$ of  ${\cal M}_a^\ASD(E)_\lambda$.
\end{re}

\subsection{The construction of the isomorphism}

Put 
$${\cal G}_E^0:=\left\{f=\left(\begin{matrix}
f_{11} & f_{12}\\
f_{21} & f_{22}
\end{matrix}
\right)\in {\cal G}^\C_E\left| \ \int_X (f_{11}- 1)=0\right.\right\}
$$ 
${\cal G}_E^0$ is not a subgroup of ${\cal G}^\C_E$; it is a complex hypersurface which is transversal to  the subgroup 
$\left\{\left(\begin{matrix} z &0 \\ 0& z^{-1}\end{matrix}\right)\vline\ z\in \C^*\right\}\simeq \C^*$ at $\id_E$, and is invariant under the inner action of the  subgroup
\begin{equation}\label{IdSubgr}\left\{\left(\begin{matrix}\phi &0 \\ 0& \phi^{-1}\end{matrix}\right)\vline\ \phi\in {\cal C}^\infty(X,\C^*)\right\}\simeq {\cal G}^\C\ .
\end{equation}
Let
$$S:{\cal A}(L)\times A^1(L')\times {\cal G}_E^0\to   A^0(L') \oplus A^0(L') \oplus A^0(X,i\R)_{r}\oplus A^0(X,i\R)_{r}$$
 be the map defined by 
$$S(b,\beta,f):=\Sg(f\cdot (b,\beta))\ ,
$$
where on the right we used the ${\cal G}^\C_E$-action on ${\cal A}(L)\times A^1(L)$ induced via $\theta$ and (\ref{IdSubgr}) from the standard action of ${\cal G}^\C_E$ on ${\cal A}_a(E)$ (see \cite{D}).  We recall that the latter action is induced from the standard ${\cal G}^\C_E$-action on ${\cal A}^{0,1}_\delta(E)$ (see section \ref{semi}) via the real isomorphism of affine spaces
$$ j:{\cal A}_a(E)\to {\cal A}^{0,1}_\delta(E)\ ,\ A\mapsto \bar\partial_A\ .
$$ 

\begin{pr} \label{D} The map $S$ has the following properties
\begin{enumerate}[1.]
\item Is $G$-equivariant with respect to the action
$$\varphi\cdot (\beta_1,\beta_2,u_1, u_2)=(\varphi^2 \beta_1,\varphi^2\beta_2,u_1,u_2)$$
of $G$ on $A^0(L') \oplus A^0(L') \oplus A^0(X,i\R)_{r}\oplus A^0(X,i\R)_{r}$.
\item \label{derivative} For every $b\in {\cal A}(L)$ one has 
$$\frac{\partial S}{\partial f}(b,0,\id_E)=P_b:=\delta_b d_b\ ,$$
 where
$$d_b:T_{\id_E}{\cal G}_E^0=A^0(X,\C)_r\times A^0(L')\times A^0(L'')\to  A^1(X,i\R)\times A^1(L')\ ,
$$
$$\delta_b: A^1(X, i\R)\times A^1(L')\to  A^0(L')\times A^0(L')\times A^0(X,i\R)_r\times A^0(X,i\R)_r 
$$
are first order operators given by
$$d_b(f,v',v''):=\left(
\begin{array}{c}
-\partial \bar f+ \bar\partial f\\
-\partial_{b'} (v'')^*+ \bar\partial_{b'}v'
\end{array}
\right),\ \delta_b(\dot b,\dot\beta):=\left(\begin{array}{c}
\Lambda_g (\partial_{b'} \dot\beta^{01} - \bar\partial_{b'}\dot\beta^{10})  \\
 \Lambda_g (\partial_{b'}\dot\beta^{01} +  \bar\partial_{b'}{\dot \beta}^{10}) \\
p_r \Lambda_g d\dot b\\
p_r\Lambda_g d^c\dot b
\end{array}\right) .
$$
This composition is an elliptic second order operator, and is an isomorphism when $b\in b_0 + H_{\rm cl}$.

\item There exists an open, $G$-invariant neighborhood ${\cal U}$ of $\left\{b_0+ H_{\rm cl}\right\}\times\{0\}$
in  ${\cal A}(L)\times A^1(L')$, and an open, $G$-invariant neighborhood ${\cal V}$ of $\id_E$ in ${\cal G}_E^0$ such that the intersection of the zero locus $Z(S)$ with ${\cal U}\times {\cal V}$ is the graph of a smooth, $G$-equivariant function $r: {\cal U}\to  {\cal V}$
satisfying 
$$\resto{r}{{\cal U}\cap (b_0+H_{\rm cl})}\equiv \id_E.$$
\end{enumerate}

\end{pr}

The first two claims can be checked by direct computations. The third claim follows from the first and the second. The map $r$ is obtained locally, around  the points of $\left\{b_0+ H_{\rm cl}\right\}\times\{0\}$, by applying the implicit function theorem to   $S$.
\qed

\vspace{2mm}

Using $r$ we obtain a $G$-equivariant map $R:{\cal U}\to ({\cal A}(L)\times A^1(L'))^\Sg$ given by
$$ R(b,\beta):= r(b,\beta)\cdot (b,\beta)\ .
$$
By construction, this map has the remarkable property
\begin{re} \label{remarkable}   For any $(b,\beta)\in {\cal U}$ the connections $j(\theta(b,\beta))$, $j(\theta(R(b,\beta)))\in {\cal A}^{0,1}_\delta(E)$  belong to the same ${\cal G}^\C_E$-orbit.
\end{re}
Using Proposition \ref{D} (\ref{derivative}) we see that, for  a point $b\in b_0+ H_{\rm cl}$, one has 
  \begin{equation}\label{DR}
 D_{(b,0)} R=\p_{\ker(\delta_b)} \ ,
  \end{equation}
where $\p_{\ker(\delta_b)}$ stands for the projection $A^1(X,i\R)\times A^0(L')\to \ker(\delta_b)$ associated with the direct sum decomposition
 $ A^1(X, i\R)\times A^1(L')=\im (d_b)\oplus  \ker(\delta_b)$.\\

We will compare now the   moduli spaces ${\cal N}$, ${\cal N}^{\Sg\Ig}$ using the
$\R$-affine, $G$-equivariant embedding %
$$\iota:{\cal C}\to {\cal A}(L)\times A^0(L'),\ \iota(v,\sigma,\alpha',\alpha''):=\left(\begin{array}{c}
(\sigma+\partial_\sigma)+ (-\bar v+ v)\\
\alpha'-(\alpha'')^*
\end{array}\right)\ ,$$
 where $\partial_\sigma$ denotes the unique operator $A^0(L)\to A^{1,0}(L)$ for which $\sigma+\partial_\sigma$ is a Hermitian linear connection on $L$. The pull-back of the ``integrability" system $(\Ig)$ via $\iota$ is precisely the system $(\Ng)$ involved in the definition of ${\cal N}$, therefore $\iota$ induces a map ${\cal C}^\Ng\to ({\cal A}(L)\times A^1(L'))^{\Ig}$ which will be denoted by the same symbol $\iota$.  Note now that, since   $(\Ig)$ is ${\cal G}_E^\C$-invariant, the composition 
$$\Rg:=\resto{R\circ \iota}{\iota^{-1}({\cal U})\cap {\cal C}^\Ng}:{\iota^{-1}({\cal U})}\cap {\cal C}^\Ng\to ({\cal A}(L)\times A^1(L'))^{\Sg}
$$
also takes in fact values in ${\cal A}(L)\times A^1(L')^{\Sg\Ig}$. 
Using (\ref{DR}) and  the identification (\ref{TCNg}) we get

$$D_{(0,\sigma,0,0) }\Rg (\dot\sigma,\dot \alpha', \dot \alpha'')=
\left(\begin{array}{c} -\overline{\dot\sigma}+ \dot\sigma\\
-(\dot \alpha'')^*+\dot \alpha' \end{array} \right) \ \forall \sigma\in \sigma_0+H^{0,1}_{\rm cl}\ ,
$$
for any $\sigma\in \sigma_0+H^{0,1}_{\rm cl}$, $\dot\sigma\in H^{0,1}$, $\dot \alpha'\in \Hg'_\sigma$, $\dot \alpha''\in \Hg''_\sigma$.  This shows that the restriction of  $\Rg$ to a sufficiently small open neighborhood of $T$ (which can be supposed to be $G$-invariant) is a diffeomorphism. Taking $G_{x_0}$-quotients we obtain   a smooth, $S^1$-equivariant map
\begin{diagram}[h=6mm]
\qmod{\iota^{-1}({\cal U})\cap {\cal C}^\Ng}{G_{x_0}}&\rTo^{\rg} &{\cal N}^{\Sg\Ig}\\
\dInto & &\\
{\cal N} &&
\end{diagram}
induced by $\Rg$ and defined on an open, $S^1$-invariant neighborhood of $T$ in ${\cal N}$. 
Let now 
$${\cal W}\subset \qmod{\iota^{-1}({\cal U})\cap {\cal C}^\Ng}{G_{x_0}}$$
 be a sufficiently small, open, $S^1$-invariant   neighborhood  of $T$ in  ${\cal N}$ such that 
 \begin{itemize}
 \item The restriction $\rg_{\cal W}:=\resto{\rg}{{\cal W}}$ is a diffeomorphism on its image ${\cal N}_{{\cal W}}^{\Sg\Ig}:=\rg({\cal W})$.
 \item ${\cal W}\cap {\cal N}^{\Sg\Ig\Mg}\subset {\cal N}^{\Sg\Ig\Mg}_\epsilon$, where $\epsilon>0$ satisfies the property stated in Proposition \ref{theta}, 
 \item  ${\cal W}$  satisfies  the property stated in Remark \ref{isoiso}, 
 \item  the image $W$  of ${\cal W}$ in ${\cal N}/\C^*$  satisfies  the property stated in Remark \ref{eta}.
 \end{itemize}

Consider now the following smooth, $S^1$-invariant $\R$-valued maps 
$$\mg_{\cal W}=\resto{\mg}{{\cal N}_{{\cal W}}^{\Sg\Ig}}:{\cal N}_{{\cal W}}^{\Sg\Ig}\to\R ,\ \psi:=\mg_{\cal W} \circ \rg_{\cal W}: {\cal W}\to\R,\ \varphi:=\mg_{\cal W} \circ \rg_{\cal W}\circ \bg_{\cal W}: \ag({\cal W})\to\R.
$$
(see Remark \ref{isoiso} for the notations $\ag_{\cal W}$, $\bg_{\cal W}$, $\Ag_{\cal W}$, $\Bg_{\cal W}$).
Using the functoriality of the spherical blowup with respect to diffeomorphisms \cite{AK}, we obtain  diffeomorphisms of manifolds with boundary
$$ \widehat{Z(\varphi)}  \textmap{\hat\bg_{\cal W}}  \widehat{Z(\psi)} \textmap{\hat \rg_{\cal W}}   \widehat{Z(\mg_{\cal W})}  \ ,\ 
\qmod{\widehat{Z(\varphi)}}{S^1} \textmap{[\hat\bg_{\cal W}]}  \qmod{\widehat{Z(\psi)}}{S^1}\textmap{[\hat \rg_{\cal W}]}  \qmod{ \widehat{Z(\mg_{\cal W})}}{S^1 } \ ,  $$
extending the diffeomorphisms $\bg_{\cal W}^*$, $\rg_{\cal W}^*$, $[\bg_{\cal W}^*]$, $[\rg_{\cal W}^*]$ between the corresponding interiors. In particular, using Remark \ref{blowup-embed}, we get an open embedding
\begin{equation}\label{blu} \widehat{[\theta]}\circ [\hat\rg_{\cal W}]\circ [\hat\bg_{\cal W}]:\widehat{Z(\varphi)}/S^1 \hookrightarrow \hat {\cal M}^\ASD_a(E)_\lambda\ .
\end{equation}
\vspace{3mm}
\begin{proof} (of Theorem \ref{hol-model})  First of all note that rescaling the metric $h$ on $E$ if necessary we may suppose that the Chern connection $a$ of the pair $({\cal D},\det(h))$ is Hermite-Einstein, hence the formalism developed in sections \ref{a-flip-passage}, \ref{interm} applies.

Using Lemma \ref{P1P2P3} below and  Proposition \ref{prop} it follows that (choosing a smaller ${\cal W}$ if necessary) the blowup $S^1$-quotient $\widehat{Z(\varphi)}/S^1$ is identified  with a neighborhood of $\partial \hat Q$ in the blowup flip passage $\hat Q$ associated with this system. Therefore, using (\ref{blu}), we get an open embedding $ \hat Q\supset O\stackrel{\chi}\hookrightarrow {\cal O}_\lambda\subset  \hat {\cal M}^\ASD_a(E)_\lambda$ satisfying the first claim of Theorem \ref{hol-model}. \\ 
\\
For the second claim of Theorem \ref{hol-model} we have to check that the composition 
$$Z(\varphi)^*/S^1\textmap{[\theta^*]\circ[\rg_{\cal W}^*]\circ[\bg_{\cal W}^*]} {\cal M}^\ASD_a(E)^*$$
becomes holomorphic  if one endows $Z(\varphi)^*/S^1$ with the holomorphic structure induced by the embedding $Z(\varphi)^*/S^1\hookrightarrow  \Hg^*/\C^*$, and ${\cal M}^\ASD_a(E)^*$ with the holomorphic structure induced by the embedding $[j^*]:{\cal M}^\ASD_a(E)^*\hookrightarrow {\cal M}^\si_{\cal D}(E)$.  
 The required holomorphy property follows by Remarks \ref{eta} (the holomorphy of $[\eta^*]$), \ref{etale} (the holomorphy and étale property of $\{\ag^*\}$), using the commutative  diagram

\begin{diagram}[h=8mm]
\qmod{Z(\varphi)^*}{S^1}&\lTo^{\ \ [\ag_{\cal W}^*]\simeq\ \  } &\qmod{Z(\psi)^*}{S^1} &\rTo^{\ \ \ [\rg_{\cal W}^*]\simeq\  \ \ } &\qmod{Z(\mg_{\cal W})^*}{S^1} &\rInto^{\ \ \ [\theta^*]\ \ \ } & {\cal M}^\ASD_a(E)^* &&&\\
\dInto & &\dInto & &&&\dInto^{[j^*]}\\
\Hg^*/\C^*& \lTo^{\{\ag^*\}}&W\setminus {\cal N}^{\C^*} & &\rInto ^{[\eta^*] } & &{\cal M}^\si_{\cal D}(E)\ .
\end{diagram}
\vspace{2mm}\\
The key ingredient in the proof is the commutativity of the right hand rectangle, which follows from Remark \ref{remarkable}:
$$[j^*]([\theta^*]([\rg^*_{\cal W}]([v,\sigma,\alpha',\alpha'']))=[(j\circ\theta)(R(\iota(v,\sigma,\alpha',\alpha'')))]=[(j\circ\theta)(\iota(v,\sigma,\alpha',\alpha''))]$$
$$=[\eta(v,\sigma,\alpha',\alpha'')]=[\eta^*]([v,\sigma,\alpha',\alpha''])\ .
$$
\end{proof}
\begin{lm}\label{P1P2P3}
The map $\varphi$ satisfies the properties P1, P2, P3 of section \ref{pert} written for the system 
$$\left({\cal M}(L)_\varepsilon,\resto{{\cal H}'}{{\cal M}(L)_\varepsilon}, \resto{{\cal H}''}{{\cal M}(L)_\varepsilon}, \hg', \hg'', \resto{f_{\cal D}}{{\cal M}(L)_\varepsilon}\right)$$
 defined in section \ref{a-flip-passage}.
\end{lm}

\begin{proof} P1 and P2 follow  immediately  from the formula
\begin{equation}\label{Dmg}
(D_{(b,0)}\Mg )(\dot b,\dot \beta)=\frac{1}{2}\int i\Lambda_g ( d \dot b)\vol_g\ \forall b\in b_0+ H_{\rm cl}\ \forall \dot b\in A^1(X,i\R)\ \forall \dot \beta\in \Hg_b\ .
\end{equation}
P3 is more delicate, because we have to compute the second derivative  of a complicated composition of non-linear maps, including  $\rg_{\cal W}$ which is induced by $R$ (defined  implicitly via $r$, not explicitly).   The result follows from the following formulae, whose proofs will be omitted:
\begin{enumerate}

\item Let $b\in b_0+H_{\rm cl}$ and $\beta\in A^1(L')$. Putting $\xi=(0,\dot \beta)$ we have
$$d_b P_b^{-1}\big((D^2_{(b,0)}\Sg)(\xi,\xi)\big)=\big(i d^c Q^{-1}p_r\big\{ |\dot \beta^{01}|^2-|\dot \beta^{10}|^2\big\}, 0\big)\ ,
$$
where $Q:=p_r\Lambda_g d^c d:A^0(X,\R)_r\textmap{\simeq} A^0(X,\R)_r$ (see section \ref{Hodge}).

\item Let $b\in b_0+H_{\rm cl}$ and let $(x_t)_{t\in (-\varepsilon,\varepsilon)}$ be  a smooth path in ${\cal A}(L)\times A^1(L')$ such that $x_0=(b,0)$,  $\dot x_0\in \ker(\delta_b)$. Then
\begin{enumerate}
\item 
$
\at{\frac{d^2}{dt^2}}{t=0}R(x(t))=\p_{\ker(\delta_b)}(\ddot x_0)- d_b\big (P_b^{-1}\big((D^2_{x_0}\Sg)(\dot x_0,\dot x_0)\big)\big)$.
\item If $\dot x_0=(0,\dot \beta)$ with $\dot\beta\in \bar \Hg''_b\oplus \Hg'_b$ and 
$\ddot x_0\in V\oplus (\bar \Hg''_b\oplus \Hg'_b)$,
then 
$$\at{\frac{d^2}{dt^2}}{t=0}\Mg(R(x(t)))=\int_X (|\dot \beta^{01}|^2-|\dot \beta^{10}|^2)\vol_g\ .
$$
\end{enumerate}
\end{enumerate}

Note that, for obtaining the last formula, we used the important cancellation
\begin{equation}\label{Gau}\int \Lambda_g d d^c \big(Q^{-1}p_r\big\{ |\dot \beta^{01}|^2-|\dot \beta^{10}|^2\big\}\big)\vol_g=0\ ,
\end{equation}
which follows from the Gauduchon condition. To complete the proof one chooses $\sigma\in\Sigma_0$ and applies (\ref{Gau}) to   $x_t=\iota(\Bg(u_t))$, where $(u_t)$ is a path in $\Hg_\sigma$ with $u_0=0$.
 
\end{proof}

\section{Appendix}

\subsection{Hodge type decomposition surfaces on Gauduchon surfaces}\label{HodgeSection}

In this section a {\it complex surface} is a compact, connected 2-dimensional complex manifold. Let $X$ be a complex surface. We   denote by $A^0(X,\R)_r$ (respectively $A^0(X,\C)_r$) the kernel of the operator
$\int_X$ on the space  $A^0(X,\R)$  (respectively $A^0(X,\C)$), i.e., the $L^2$-orthogonal complements of the   line of constants in this space. We will also denote by  the same symbol $p_r$ the  $L^2$-orthogonal projections 
$$ A^0(X,\R)\textmap{p_r} A^0(X,\R)_r\ ,\  A^0(X,i\R)\textmap{p_r} A^0(X,i\R)_r\ ,\  A^0(X,\C)\textmap{p_r} A^0(X,\C)_r\ .$$
The operator     
$$Q:=\resto{p_r\Lambda d^c d}{A^0(X,\R)_r}:A^0(X,\R)_r\to A^0(X,\R)_r$$
 is an isomorphism (see Proposition 1.2.8 p. 33 \cite{LT1}), and the subspace $\ker (p_r\Lambda_g d^c)$ of $A^1(X,\R)$ is a topological complement of $d(A^0(X,\R))$ in $A^1(X,\R)$. 

 \begin{thry} \label{Hodge} Let $(X,g)$ be a Gauduchon surface.   Put
  $$H:=\{a\in i A^{1}(X)|\ \bar\partial a^{01}=0,\ p_r\Lambda_g da=0,\  p_r\Lambda_g d^ca=0 \}
$$
$$H^{0,1}:=\{\alpha\in   A^{0,1}(X)|\ \bar\partial\alpha=0,\ p_r\Lambda_g\partial\alpha=0\}\subset Z_{\bar\partial}^{01}(X)\subset A^{0,1}(X)\ ,
$$
$$H_{\rm cl}:=\{a\in i A^{1}(X)|\   da=0,\  p_r\Lambda_g d^ca=0 \}\ ,
$$
$$H_{\rm cl}^{0,1}:=\{\alpha\in   A^{0,1}(X)|\ d(-\bar\alpha+\alpha)=0,\   p_r\Lambda_g\partial\alpha= 0 \}\ ,
$$
$$V^{0,1}:=\{\alpha\in A^{0,1}(X)|\ \Lambda_g \partial \alpha=0\}\ ,\ V:=\{a\in i A^1(X,\R)|\ \Lambda_g d a=\Lambda_g  d^c a=0\}\ . $$
 \begin{enumerate}
 \item \label{morphisms} The natural morphisms 
$$H^{0,1}\to H^{0,1}_{\bar\partial}(X)=H^1(X,\cal{O}_X)\ ,\ H_{\rm cl}\to H^1_{\rm DR}(X,i \R)$$
are  isomorphisms. 
\item  Denoting by $J$ the obvious complex structure on $H$, the map $a\mapsto a^{01}$ induces a complex isomorphism $(H,J)\to H^{0,1}$ which restricts to a real isomorphism $H_{\rm cl}\to H_{\rm cl}^{0,1}$.
\item $H_{\rm cl}$ is the kernel of the linear functional $H\to\R$ defined by $a\mapsto \int i \omega_g \wedge da$.  This functional is  trivial if and only if $b_1(X)$ is even.
\item Denoting by $J$ the obvious complex structure on $V$, the map $v\mapsto v^{01}$ induces a complex isomorphism $(V,J)\to V^{0,1}$.
\item When $b_1(X)$ is even, then $H^{0,1}\subset V^{0,1}$ and  $H\subset V$. When $b_1(X)$ is odd then $H^{0,1}\cap V^{0,1}$  (respectively $H\cap V$) has complex codimension 1 in $H^{0,1}$ (respectively in $(H,J)$). 

\item If $b_1(X)=1$ then $H^{0,1}\cap V^{0,1}=\{0\}$, $H\cap V=\{0\}$, and one one has direct sum decompositions
 $$A^{0,1}(X)=H^{0,1}\oplus V^{0,1}\oplus\bar\partial(A^0(X,\C)) \ ,
 $$
 $$i A^1(X,\R)=H\oplus V \oplus  d^c(A^0(X,i \R)) \oplus d(A^0(X,i \R))  \ .
 $$

 \end{enumerate}
 
 \end{thry}

We also mention the following important
\begin{pr}\label{H1}
Let ($X,g)$ be a connected Gauduchon compact complex manifold, and let $({\cal E},h)$ be a Hermitian holomorphic bundle on $X$ whose Chern connection $A$ is Hermite-Einstein with vanishing Einstein constant.
Then
\begin{enumerate} [1.]
\item $\ker(d_A)=\ker(\bar\partial_A)=H^0({\cal E})$,
\item If $\ker(d_A)=\{0\}$ then 
\begin{enumerate}
\item $\ker \Lambda_g \partial_A$ is complement of $\bar\partial_A(A^{0}(E))$ in $ A^{0,1}(E)$.
\item The space  $H^{0,1}_A:=\{\alpha\in A^{0,1}(E)|\ \bar\partial_A(\alpha)=0,\ \Lambda_g \partial_A \alpha=0\}$
is identified with $H^1({\cal E})$ via the obvious morphism.
\end{enumerate}
\end{enumerate}
\end{pr}

\begin{proof}
Since $\Lambda_g F_A=0$ we get $\Lambda\partial_A \bar\partial_A=- \Lambda\bar \partial_A  \partial_A$.   Using the maximum principle for the operator $i\Lambda_g\bar\partial\partial$ as in the proof of Theorem 2.2.1 p. 50 \cite{LT1} one can prove that $\ker(\Lambda_g \partial_A \bar\partial_A)=\ker d_A$, which proves 1. If we assume that $\ker(d_A)=\{0\}$, then  the operator $\Lambda_g \partial_A \bar\partial_A$ will be injective hence, since it it has vanishing index, it will be an isomorphism. It follows that for any $\alpha\in A^{0,1}(E)$ there exists a unique section $\varphi\in A^0(E)$ such that $\Lambda_g \partial_A (\alpha+\bar\partial_A\varphi)=0$. This proves 2. The third statement follows from 2. and the Dolbeault theorem.
 \end{proof}
\subsection{Integrable semiconnections}
\label{semi}

Let $X$ be a compact connected complex manifold, and $E$ a  differentiable complex vector bundle of rank $r$ on $X$. We recall that a semiconnection on $L$ is a first order differential operator $\eta:A^0(E)\to A^{0,1}(E)$
satisfying the Leibniz rule
$$\eta(\varphi s)=(\bar\partial \varphi)\otimes  s+\varphi \eta(s)\ \forall \varphi\in A^0(X,\C)\ \forall s\in A^0(E)\ .
$$
We denote by $ {\cal A}^{0,1}(E)$ the space of semiconnections on $E$; this space has the structure of an affine space  with model vector space $A^{0,1}(\End(E))$. A semiconnection $\eta$  on $E$ admits natural extensions $A^{0,q}(E)\to A^{0,q+1}(E)$ satisfying the obvious Leibniz rule, and which will be denoted by the same symbol $\eta$. In particular one can consider the composition $\eta\circ \eta:A^0(E)\to A^{0,2}(E)$, which is 0-order operator, hence it is given by multiplication with a form $F_\eta^{02}\in A^{0,2}(\End(E))$. A semiconnection $\eta$ on $E$ is called integrable if $F^{02}_\eta=0$. If this is the case, then $\eta$ defines a holomorphic structure  ${\cal E}_\eta$ on $E$. We will use the same symbol for the corresponding holomorphic vector bundle. The corresponding locally free sheaf on $X$ is just the  sheaf of germs of local sections $s$ of $E$ satisfying the equation $\eta(s)=0$. The assignment $\eta\mapsto {\cal E}_\eta$
defines a bijection between the space $  {\cal A}^{0,1}(E)^{\mathrm{int}}$ of integrable semiconnections on $E$ and the space ${\cal H}(E)$ of holomorphic structures on $E$.  The   group $\Gamma(X,\GL(E))$ acts naturally on $ {\cal A}^{0,1}(E)$  by the formula
$$\varphi\cdot \eta:=\varphi\circ \eta\circ \varphi^{-1}=\eta-(\eta \varphi )\varphi^{-1}\ .
$$

Fix an integrable semiconnection $\delta$ on   $D:=\det(E)$, and denote by ${\cal D}$ the corresponding holomorphic structure on $D$. The subspace   
$$ {\cal A}_\delta^{0,1}(E):=\{\eta\in   {\cal A}^{0,1}(E)|\ \det(\eta)=\delta\}\subset {\cal A}_\delta^{0,1}(E)$$
is closed and invariant under the action of the gauge group ${\cal G}^\C_E:=\Gamma(X,\SL(E))$. 
The space ${\cal H}_{\cal D}(E)$ of holomorphic structures on $E$ which induce ${\cal D}$ on $\det(E)$ appearing in the definitions of the moduli spaces ${\cal M}^\pst_{\cal D}(E)$, ${\cal M}^\st_{\cal D}(E)$, ${\cal M}^\si_{\cal D}(E)$ of section \ref{intro2} can be identified with the subspace
$$
 {\cal A}_\delta^{0,1}(E)^{\mathrm{int}}:=\{\eta\in   {\cal A}^{0,1}(E)^{\mathrm{int}}|\ \det(\eta)=\delta\}\subset {\cal A}^{0,1}(E)^{\mathrm{int}}\ .$$
\subsection{The gauge theoretical Picard group of a compact complex manifold}\label{GaugePic}

Let $X$ be a compact connected complex $n$-manifold, and $L$ a  differentiable complex line bundle on $X$. The space ${\cal A}^{0,1}(L)$ of semiconnections on $L$ is an affine space with model vector space $A^{0,1}(X)$, which is independent of $L$. The obstruction $F_\sigma$ to integrability of a semiconnection $\sigma\in   {\cal A}^{0,1}(L)$ is an element of the space $A^{0,2}(X)$, which is also independent of $L$. Note that  
$$  {\cal A}^{0,1}(L)^{\rm int}\ne\emptyset\Leftrightarrow c_1(L)\in\NS(X)\ ,
$$
where $\NS(X)$ is the kernel of the obvious morphism $H^2(X,\Z)\to H^2(X,{\cal O}_X)$. 

 For a semiconnection $\sigma\in   {\cal A}^{0,1}(L)$ and a form $v\in A^{0,1}(X)$ one has
$$F_{\sigma+ v}^{02}=F_\sigma^{02}+\bar\partial v \ ,
$$
which shows that the  subspace $  {\cal A}^{0,1}(L)^{\rm int}\subset   {\cal A}^{0,1}(L)$ of integrable semiconnections on $L$, {\it if non-empty}, is an affine subspace of  ${\cal A}^{0,1}(L)$ whose model vector space is the space  $Z^{0,1}_{\bar\partial}(X)$ of $\bar\partial$-closed $(0,1)$-forms. We will denote by ${\cal L}_\sigma$ the holomorphic structure on $L$ (and the holomorphic line bundle)  defined by an integrable semiconnection $\sigma$.
The natural action of the complex gauge group ${\cal G}^\C:={\cal C}^\infty(X,\C)$ on  $ {\cal A}^{0,1}(L)$  is given by the formula
$$\varphi\cdot \sigma:=\varphi\circ \sigma\circ \varphi^{-1}=\sigma-(\bar\partial \varphi )\varphi^{-1}\ .
$$
and the stabilizer of any point $\sigma\in   {\cal A}^{0,1}(L)$ with respect to this action is $\C^*$, so we obtain an induced free action of the reduced complex gauge group ${\cal G}^\C_0:= {{\cal G}^\C}/{\C^*}$. 
This gauge action leaves invariant the closed subspace $ {\cal A}^{0,1}(L)^{\rm int}$, and, assuming  $c_1(L)\in\NS(X)$, the quotient 
$${\cal M}(L):=\qmod{ {\cal A}^{0,1}(L)^{\rm int}}{{\cal G}^\C}
$$
is a $h^{1}({\cal O}_X)$-dimensional dimensional complex manifold which can be identified with the connected component $\Pic^{c_1(L)}(X)$ of the Picard group $\Pic(X)$. This manifold can be regarded as a subspace of the infinite dimensional quotient
$$ {\cal B}^{0,1}(L):=\qmod{  {\cal A}^{0,1}(L)}{{\cal G}^\C}\ .
$$

For a fixed point $x_0\in X$ denote by ${\cal G}^\C_{x_0}$ the kernel of the evaluation morphism $\ev_{x_0}: {\cal G}^\C\to\C^*$. This group is naturally isomorphic to ${\cal G}^\C_0$ and acts freely on $  {\cal A}(L)$. The quotient
$$ {\cal A}^{0,1}(L)^{\mathrm{int}}\times_{{\cal G}^\C_{x_0}} L
$$
can be naturally regarded as a line bundle on the product ${\cal M}(L)\times X$. This line bundle comes with a tautological holomorphic structure and will be denoted by ${\cal L}_{x_0}$. Via the identification ${\cal M}(L)=\Pic^{c_1(L)}(X)$ this line bundle corresponds to the Poincaré line bundle normalized at $x_0$.\\

Endow  $L$ with a Hermitian metric $h$. The gauge group of the Hermitian line bundle $L$ is ${\cal G}:={\cal C}^\infty(X,S^1)$. We recall that a Hermitian connection $b$ on $(L,h)$  is called Hermitian-Einstein if is satisfies the equations
$$F_b^{02}=0\ ,\ p_r[\Lambda_g F_{b}]=0 \ .
$$

Note that the second  condition $p_r[\Lambda_g F_{b}]=0$ is equivalent to the classical Hermitian-Einstein condition ``$i\Lambda_g F_{b}$ is a constant" (called the Einstein constant of the connection). In the non-Kählerian framework the map which assigns to a holomorphic line bundle the Einstein constant of a compatible Hermite-Einstein connection  is not necessarily constant  on $\Pic^c(X)$, because the degree map associated with a Gauduchon metric is not a topological invariant in general \cite{LT1}.

The space ${\cal A}^{\rm HE}(L)\subset{\cal A}(L)$ of Hermite-Einstein connections on $L$  is an affine subspace with model vector space 
$$\cal{H}:=\{a\in A^1(X,i\R)|\ \bar\partial a^{01}=0,\ p_r\Lambda d a=0\} \ .
$$
Therefore,  fixing a Hermite-Einstein connection $b_0\in {\cal A}^{\rm HE}(L)$, we have %
$${\cal A}^{\rm HE}(L)=b_0+ \cal{H}\ .$$
As shown in \cite{LT1} in gauge theory it is convenient to replace the usual ``Coulomb slice condition" $d^*a=0$ on 1-forms by the condition $p_r\Lambda_g d^c a=0$. The two conditions are equivalent in the Kählerian case; in the general Gauduchon case they both define slices for the action of $\cal{G}$ on ${\cal A}(L)$. The advantage of the new slice condition introduced in \cite{LT1} is that the intersection 
$$H:=\{a\in i A^{1}(X)|\ \bar\partial a^{01}=0,\ p_r\Lambda_g da=0,\  p_r\Lambda_g d^ca=0 \}
$$
of $\cal{H}$ with $\ker(p_r\Lambda_g d^c)$ is $J$-invariant, so comes with a natural complex structure.  
We define the subgroups
$$G:=\{\varphi\in{\cal G}|\ p_r\Lambda d^c(\varphi^{-1} d\varphi)=0\}\ ,\ G^\C:=\{\varphi\in{\cal G}^\C|\ p_r\Lambda_g \partial (\varphi^{-1}\bar\partial \varphi)=0\} 
$$
of ${\cal G}$ and ${\cal G}^\C$ respectively.

\begin{pr} \label{groups} Let $X$ be a compact complex surfaces endowed with a Gauduchon metric $g$. \begin{enumerate}
\item One has $G^\C=G\times\R_{>0}$.
\item The map ${\cal G}^\C\to   H^1(X,\C)$ given by $\varphi\mapsto [\varphi^{-1} d\varphi]_{\rm DR}$ takes values in the group $2\pi i H^1(X,\Z)$ and induces
\begin{enumerate}
\item An epimorphism $q:G\to 2\pi i H^1(X,\Z)$ and a short exact sequence
$$\{1\}\to S^1\to G\to   2\pi i H^1(X,\Z)\to \{1\}\ ,
$$

\item  An epimorphism $q^\C:G^\C\to 2\pi i H^1(X,\Z)$ and a short exact sequence
$$\{1\}\to \C^*\to G^\C\to   2\pi i H^1(X,\Z)\to \{1\}\ .
$$
\end{enumerate}
\end{enumerate}
\end{pr}
\begin{proof}
(1) For $\varphi\in {\cal G}^\C$ write locally $\varphi=e^f$ for a (locally defined) smooth complex function $f$ and note that 
$$\partial (\varphi^{-1}\bar\partial \varphi)=\partial \bar\partial f\ ,\ \partial (\bar\varphi^{-1}\bar\partial \bar \varphi)=\partial \bar\partial \bar f\ ,\ \partial ((\varphi\bar\varphi)^{-1}\bar\partial (\varphi\bar\varphi))=\partial \bar\partial (f+\bar f)\ .$$
 Since $i\partial \bar\partial$, these formulae show that 
$$p_r\Lambda_g \partial (\varphi^{-1}\bar\partial \varphi) =0\Rightarrow  p_r\Lambda_g  \partial ((\varphi\bar\varphi)^{-1}\bar\partial (\varphi\bar\varphi))=0 \Rightarrow p_r\Lambda_g  d^c d\log|\varphi|^2=0\ ,
$$
which implies that $|\varphi|$ is constant. Therefore any element $\varphi\in G^\C$ can be written as $ \varphi= c e^\psi$ where $c\in\R_{>0}$ and $\psi\in {\cal G}$. Writing locally $\psi= e^g$ for a pure imaginary (locally defined) function we get $\Lambda_g\partial \bar\partial g=0$, which is equivalent to $\Lambda_g d^c d g=0$. This implies $\Lambda_g d^c(\psi^{-1} d\psi)=0$, hence $\psi\in G$.
\\ \\
(2)
The Cauchy formula shows that  the Rham cohomology class of the complex 1-form $\frac{1}{2\pi i}z^{-1} dz$ is the canonical generator $\gamma$ of $H^1(\C^*,\Z)$. Therefore for any $\varphi\in {\cal G}^\C$ one has
$$[\varphi^{-1} d\varphi]_{\rm DR}=2\pi i\varphi^*(\gamma)\in 2\pi i H^1(X,\Z)\ .
$$
(a) Since $S^1$ is a $K(\Z,1)$-space it follows easily that the map
$$[X,S^1] \to  H^1(X,\Z)
$$ 
given by $[\varphi]\mapsto \varphi^*(\gamma)$ is an isomorphism. Therefore for every $y\in 2\pi i H^1(X,\Z)$ there exists $\varphi_y\in {\cal G}$ with $[\varphi_y^{-1} d\varphi_y]_{\rm DR}=y$. It suffices to find $f\in i A^0(X,\R)$ such that putting $\varphi=e^f \varphi_y$  one has 
$p_r\Lambda_g d^c (\varphi^{-1} d\varphi)=0$.
This condition is equivalent to $p_r\Lambda_g d^c d f+ p_r\Lambda_g d^c (\varphi_y^{-1} d\varphi_y)=0$. Since the operator $Q$ is an isomorphism, this equation has a unique solution $f\in i A^0(X,\R)_r$.
\\ \\
(b) This follows from (1) and (2) (a).
\end{proof}
\begin{co} \label{kerev} Denoting $G_{x_0}:=\ker(ev_{x_0}:G \to S^1)$, $G_{x_0}^\C:=\ker(ev_{x_0}:G^\C\to \C^*)$
one has
$$G_{x_0}^\C=G_{x_0}=\{\varphi\in {\cal C}^\infty(X,S^1)|\ p_r \Lambda_g d^c (\varphi^{-1} d\varphi)=0,\ \varphi(x_0)=1\}\ .
$$

\end{co}

\begin{co}\label{moduli}
\begin{enumerate}
\item Let $b_0$ be a Hermite-Einstein connection on $L$. The embedding $b_0+H\hookrightarrow b_0+{\cal H} $ induces an isomorphism of real analytic moduli spaces
$$\qmod{ b_0+H}{G}\textmap{\simeq}  {\cal M}^{HE}(L)\ .
$$
\item  Let $\sigma_0$ be an integrable semiconnection on $L$. The embedding $\sigma_0+H^{0,1}\hookrightarrow \sigma_0+{\cal Z}^{0,1}_{\bar\partial}(X) $ induces an isomorphism of complex moduli spaces
$$\qmod{\sigma_0+H^{0,1}}{G^\C}\textmap{\simeq}  {\cal M}(L)\ .
$$
\end{enumerate}

\end{co}

The stabilizer of a  point $b=b_0+h\in b_0+H$ (respectively of a point $\sigma=\sigma_0+\chi\in \sigma_0+H^{0,1}$) with respect to the $G$-action (respectively $G^\C$-action) is $S^1$  (respectively $\C^*$). Therefore, using  Proposition \ref{groups}, we obtain the following finite dimensional descriptions of the moduli spaces:

\begin{re} \label{FinDImKH} We have natural identifications
\begin{equation}\label{quotients}
{\cal M}^{HE}(L)=\qmod{ b_0+H}{2\pi i H^1(X,\Z)}\ ,\   {\cal M}(L)=\qmod{ \sigma_0+H^{0,1}}{2\pi i H^1(X,\Z)} \ ,
\end{equation}
where $2\pi i H^1(X,\Z)$ acts on the two affine spaces via the identifications
$$\qmod{G}{S^1}=2\pi i H^1(X,\Z)\ ,\ \qmod{G^\C}{\C^*}=2\pi i H^1(X,\Z)
$$
induced by the epimorphisms $q$, $q^\C$. Choosing  $\sigma_0=\bar\partial_{b_0}$, the Kobayashi-Hitchin isomorphism 
$${\cal M}^{HE}(L)\textmap{\simeq KH}  {\cal M}(L)$$
 is induced  via the formulae (\ref{quotients})  by the isomorphism $I:H\to H^{0,1}$.  \end{re}
Put $\Sigma:=\sigma_0+H^{0,1}$. The product $\Sigma\times  L$ can be regarded as a line bundle over  $\Sigma\times X$. This line bundle comes with a tautological integrable semiconnection $\sigma_{\rm taut}$ characterized by the following conditions:
\begin{enumerate}  
\item For any $\sigma\in \Sigma$  the restriction of $\sigma_{\rm taut}$ to the bundle $\{\sigma\}\times L$ over the fiber $\{\sigma\}\times X$ coincides with $\sigma$ via the obvious identifications.
\item For any $x\in X$ the restriction of $\sigma_{\rm taut}$ to the   line bundle $\Sigma\times L_x$ over the slice $\Sigma\times \{x\}$ coincides with the standard trivial semiconnection on this trivial line bundle.
\end{enumerate}

Endowing  $\Sigma\times  L$ with the obvious product  $G^{\C}$-action, we see that $\sigma_{\rm taut}$ is $G^{\C}$-invariant. Fixing $x_0\in X$ we can regard  $\Sigma\times X$ as a principal $G_{x_0}$-bundle over the product ${\cal M}(L)\times X$, hence we can construct the associated vector bundle
$$(\Sigma\times X)\times_{G^\C_{x_0}}L=\qmod{\Sigma\times  L} {G_{x_0}}\ ,
$$
which will be regarded as a line bundle over ${\cal M}(L)\times X$.
  \begin{dt}\label{Poincare} The universal (Poincaré) line bundle normalized at a point $x_0\in X$ is the holomorphic bundle $\mathscr{L}_{x_0}$ obtained by endowing the quotient bundle $\qmod{\Sigma\times  L} {G_{x_0}}$ over ${\cal M}(L)\times X$ with the integrable semiconnection induced by $\sigma_{\rm taut}$
  \end{dt}
  \begin{re}
 Via the standard identification ${\cal M}(L)=\Pic^{c_1(L)}(X)$,  the universal bundle $\mathscr{L}_{x_0}$ coincides with the Poincaré line bundle normalized at $x_0$.
  \end{re}
  In a similar way, putting $S:=b_0+H$ one obtains a universal Hermitian connection $\A_{x_0}$ on the universal Hermitian line bundle $\L_{x_0}=\qmod{S\times L}{G_{x_0}}$. 
\begin{re}\label{metric}  Since $G_{x_0}^\C=G_{x_0}$, the former group also acts by unitary isomorphisms on the line bundle $\Sigma\times  L$, hence we obtain a Hermitian metric on the universal holomorphic  line bundle $\mathscr{L}_{x_0}$. This metric is fiberwise (in the $X$-direction) Hermitian-Einstein,    and depends only on  $g$  and the metric $h_{x_0}$ on the line $L_{x_0}$. It can obtained solving fiberwise the Hermitian-Einstein equation.
\end{re}
\end{document}